\documentclass[preprint,3p,times,sort&compress]{elsarticle}

\usepackage{amssymb}
\usepackage{amsmath}
\usepackage{booktabs}
\usepackage{url}
\usepackage{tikz}
\usepackage{gnuplot-lua-tikz}
\usepackage{algorithm}
\usepackage{algpseudocode}
\usepackage{mathrsfs} 
\newcommand{\domD}{\mathscr{D}}
\renewcommand{\pi}{\piup}

\renewcommand{\Im}{\operatorname{Im}}
\DeclareMathOperator{\arsinh}{arsinh}
\DeclareMathOperator{\artanh}{artanh}
\DeclareMathOperator{\diag}{diag}
\DeclareMathOperator{\OO}{O}
\makeatletter
\newif\ifkp@upRm
\DeclareSymbolFont{Letters}{OML}{jkp}{m}{it}
\DeclareMathSymbol{\uppartial}{\mathord}{Letters}{128}
\makeatother
\DeclareMathOperator{\E}{e}
\DeclareMathOperator{\I}{i}
\newcommand{\D}{\,\mathrm{d}}
\newcommand{\mathbd}{\boldsymbol}
\newdefinition{definition}{Definition}[section]
\newtheorem{theorem}{Theorem}[section]
\newtheorem{lemma}[theorem]{Lemma}
\newtheorem{proposition}[theorem]{Proposition}

\newtheorem{remark}{Remark}[section]
\newproof{proof}{Proof}
\numberwithin{equation}{section}

\journal{Elsevier}

\begin{document}

\begin{frontmatter}



\title{Improved algorithms for computing the matrix logarithm based on the double-exponential formula~\tnoteref{mytitlenote}}
\tnotetext[mytitlenote]{This work was partially supported by JSPS
Grant-in-Aid for Scientific Research (C) JP23K03218.}


\author[HCU]{Tomoaki Okayama\corref{cor1}} 
\cortext[cor1]{Corresponding author}
\affiliation[HCU]{organization={Hiroshima City University},
addressline={3-4-1, Ozuka-Higashi, Asaminami-ku},
city={Hiroshima},
postcode={731-3194},
country={Japan}}
\ead{okayama@hiroshima-cu.ac.jp}

\author[secondauthor]{Tsubasa Setogawa}
\affiliation[secondauthor]{organization={Daiko Techno Co., Ltd.},
addressline={Daiko Hiratsuka Building, 1-14, Higashi-Hiratsuka-cho, Naka-ku},
city={Hiroshima},
postcode={730-0025},
country={Japan}}

\begin{abstract}
Tatsuoka et al.\ proposed an efficient algorithm for computing the matrix
logarithm using the double-exponential (DE) formula. Its discretization
error decreases exponentially as the number of quadrature nodes grows,
and its nested grids allow previously computed values to be reused when
the mesh is refined. They also reported that their approach is more robust
than Gauss--Legendre quadrature for ill-conditioned matrices.
We propose two algorithms based on the DE formula. The first adaptively
refines the mesh while reusing previously computed values and requires
no matrix norm computations. The second uses a rigorous error bound to
select the number of quadrature nodes a priori. Numerical experiments
show that both methods attain lower errors than Tatsuoka et al.'s method
at comparable elapsed times over the tested ranges.
\end{abstract}




\end{frontmatter}



\section{Introduction}
In this paper, we consider numerical methods for computing the matrix
logarithm. For a square matrix $A \in \mathbb{C}^{m\times m}$, a matrix
logarithm of $A$ is a matrix $X \in \mathbb{C}^{m\times m}$ satisfying
\[
\exp(X) = A,
\]
where $\exp(X) = \sum_{k=0}^{\infty}X^k/k!$.  It is known that
infinitely many such matrices $X$ exist for any nonsingular matrix
$A$~\cite{higham08:_funct_matric}. In particular, for a matrix $A$ with
no eigenvalues on the non-positive real axis $(-\infty,0]$, there
exists a unique matrix logarithm $X$ whose eigenvalues all lie in the
strip $\{ z \in \mathbb{C} : |\Im z| < \pi\}$~\cite{higham08:_funct_matric};
this unique solution is called the
principal logarithm and is denoted by $\log(A)$.  Throughout this paper,
we restrict our attention to the computation of $\log(A)$ under the
condition that $A$ has no eigenvalues on $(-\infty,0]$.  The matrix
logarithm finds applications across various fields, including quantum
entropy calculations~\cite{Zachos2007,MullerLennert2013}, biomolecular
dynamics~\cite{Horenko}, viscoelastic
fluids~\cite{Hulsen2005}, diffusion tensor
imaging~\cite{Arsigny2006}, and machine
learning~\cite{Huang2017,Anitescu2012}.

Numerical methods for computing the matrix logarithm $\log(A)$ include
inverse scaling and squaring (ISS) algorithms based on Pad\'{e}
approximation~\cite{Al-Mohy}, arithmetic-geometric mean (AGM)
iterations~\cite{Cardoso}, and numerical integration
methods~\cite{Davies,TatsuEn,Ibanez2023}. Numerical integration methods
are attractive for parallel computation because integrand evaluations
at specified quadrature nodes are independent~\cite{TatsuEn}. They can
also compute $\log(A)\mathbd{b}$ by solving shifted linear systems
without forming $\log(A)$, thereby exploiting the sparsity of $A$
~\cite{Davies,TatsuEn}. Recent quadrature methods for the matrix logarithm
include Tatsuoka et al.'s algorithm~\cite{TatsuEn}, which applies the
double-exponential (DE) formula~\cite{takahasi74:_doubl},
and Ib\'{a}\~{n}ez et al.'s method~\cite{Ibanez2023}, which combines
Romberg quadrature with ISS.

A general integral representation of the matrix logarithm is given by
\begin{align}\label{eq1}
\log(A) = (A-I) \int_0^1 \left[(1-t)I + tA\right]^{-1} \D t,
\end{align}
where $I$ denotes the identity
matrix~\cite[Theorem 11.1]{higham08:_funct_matric}.
This integral can be approximated using Gauss--Legendre (GL) quadrature.
Aceto and Durastante~\cite{AcetoDurastante2025} derived error estimates
for this approximation, which can be used to select the number of
quadrature nodes and the number of inverse scaling and squaring steps.
Although GL quadrature is effective when $A$ is close to the identity
matrix, its convergence can be slow for ill-conditioned matrices or
when the integrand varies sharply near the endpoints~\cite{TatsuEn}.

For such challenging cases, the DE formula offers an alternative.
The DE formula combines (i) the change of variables
$t = \phi(x) = 1/(1 + \E^{-\pi \sinh x})$ with (ii) the composite
trapezoidal rule applied to the transformed integral.
The change of variables transforms Eq.~\eqref{eq1} into
\begin{equation}
\log(A) = (A - I)
\int_{-\infty}^{\infty}
\left(A + \E^{-\pi \sinh x} I\right)^{-1} \frac{\pi \cosh x}{1 + \E^{\pi \sinh x}} \D x.
\label{eq:logA-after-DEtrans}
\end{equation}
The transformed integrand decays double-exponentially as $|x|\to\infty$.
Using the error bound of Trefethen and
Weideman~\cite[Theorem 5.1]{Tre2014}, Tatsuoka et al.~\cite{Tatsu}
analyzed the convergence of the DE formula for Hermitian positive
definite matrices. They derived an exponential convergence rate
with respect to the number $M$ of quadrature nodes. Their analysis
assumes that the integral over $(-\infty,\infty)$ is truncated to a
fixed finite interval $[l,r]$ and that the resulting truncation error is
negligible. In practical implementations~\cite{TatsuEn}, the interval is
chosen so that this error is below a prescribed tolerance $\varepsilon>0$.
If the interval is too narrow, the
truncation error limits the attainable accuracy; if it is too wide, a
larger number $M$ of quadrature nodes is needed to maintain the same mesh
size.
The interval-selection procedure of Tatsuoka et al.~\cite{TatsuEn}
uses matrix norms and spectral-radius information to control the
relative truncation error, leading to matrix-dependent parameter
formulas (see Eqs.~\eqref{eq:finite_s}--\eqref{eq:finite_ab} for details).
However, increasing $M$ with $[l,r]$ fixed leaves the interval truncation
error unchanged: the approximation converges to the truncated integral,
not in general to $\log(A)$. Thus, the fixed-interval convergence-rate
analysis does not by itself establish convergence to $\log(A)$ as
$M\to\infty$ with $\varepsilon$ held fixed.

To overcome this limitation, we propose two algorithms based on the DE
formula. Both algorithms use the following approximation. For a positive
integer $n$, set the mesh size $h>0$ and the truncation number
$N\in\mathbb{N}$ as
\[
h = \frac{1}{n},
\qquad
N = \left\lceil n \cdot  \arsinh(\pi n) \right\rceil,
\]
and approximate $\log(A)$ by
\[
\log(A)
\approx
(A - I)
h \sum_{k=-N}^{N}
\left(A + \E^{-\pi \sinh (k h)} I\right)^{-1}
 \frac{\pi \cosh (k h)}{1 + \E^{\pi \sinh (k h)}}.
\]
This approximation uses $M=2N+1$ quadrature nodes. With this choice,
$h\to 0$ and $Nh\to\infty$ as $n\to\infty$, so the truncation interval
$[-Nh,Nh]$ expands as the mesh is refined.
The first, norm-free adaptive algorithm refines the mesh and truncation
interval while reusing previous evaluations, without matrix norm
computations or eigenvalue estimates. The second, error-bound-based
algorithm is formulated for Hermitian positive definite matrices and uses
an explicit overall-error bound, computable from the extreme eigenvalues
of $A$, to select $n$ a priori for a prescribed tolerance in exact
arithmetic. Our analysis proves convergence to $\log(A)$ and derives this
bound for Hermitian positive definite matrices; both results extend to
diagonalizable matrices with positive real eigenvalues.

The remainder of this paper is organized as follows.
Section~\ref{sec:Tatsuoka} reviews the method of Tatsuoka et al.\ based
on the DE formula, including its interval-selection strategy,
discretization-error analysis, and adaptive implementation with reuse
of previous evaluations.
Section~\ref{sec:our-algorithm} presents the proposed algorithms,
analyzes truncation, discretization, and total errors, and describes the
norm-free stopping criterion and error-bound-based a priori node
selection. Section~\ref{sec:numerical} compares the proposed algorithms
with Tatsuoka et al.'s method in terms of accuracy, computation time,
and termination behavior, and examines the error bound.
Section~\ref{sec:conclusion} presents concluding remarks. The appendix
provides an auxiliary estimate for tail trapezoidal sums.

\section{The computation algorithm of Tatsuoka et al.}\label{sec:Tatsuoka}

\subsection{Idea of the approximation}
\label{sec:4-2-1}
In this section, we explain the computation algorithm for the matrix logarithm by Tatsuoka et al.~\cite{Tatsu,TatsuEn}.
They initially used the DE transformation
$t=1/(1+\E^{-2\sinh x})$~\cite{TatsuEn}, whereas their subsequent
convergence analysis used
$t=1/(1+\E^{-\pi\sinh x})$~\cite{Tatsu}, which coincides with the
original double-exponential transformation introduced by Takahasi and
Mori~\cite{takahasi74:_doubl}.
We use the latter transformation consistently, including when describing
their algorithm.
Let $F_{\mathrm{DE}}(x)$ be the integrand in~\eqref{eq:logA-after-DEtrans}, namely,
\begin{equation}
F_{\mathrm{DE}}(x) = (A-I)\left(A + \E^{-\pi\sinh x}I\right)^{-1}
\frac{\pi\cosh x}{1 + \E^{\pi\sinh x}}.
\label{eq:FDE-x}
\end{equation}
For real $x$, the matrix $A+\E^{-\pi\sinh x}I$ is nonsingular because
$-\E^{-\pi\sinh x}<0$ cannot be an eigenvalue of $A$.
Because $1+\E^{\pi\sinh x}>0$ as well, $F_{\mathrm{DE}}(x)$ is well
defined for every $x\in\mathbb{R}$.
They approximate the integral in~\eqref{eq:logA-after-DEtrans} by
\begin{align}
\int_{-\infty}^{\infty} F_{\mathrm{DE}}(x)\D x
&\approx \int_{l}^{r} F_{\mathrm{DE}}(x)\D x
\approx h\left\{\frac{1}{2}F_{\mathrm{DE}}(l)
+ \sum_{k=1}^{M-2} F_{\mathrm{DE}}(l+kh)
+ \frac{1}{2}F_{\mathrm{DE}}(r)\right\},
\quad h=\frac{r-l}{M-1},
\label{eq4.3}
\end{align}
where $M$ denotes the number of quadrature nodes.
The first step truncates the integral to $[l,r]$, with endpoints chosen
to control the truncation error (Section~\ref{sec:4-2-2}); the second
applies the trapezoidal rule (Section~\ref{sec:4-2-3}). We review the
overall error in Section~\ref{subsec:total_error} and the adaptive
implementation in Section~\ref{sec:adaptive-Tatsuoka}.

\subsection{Selection of interval endpoints $l$ and $r$}\label{sec:4-2-2}


%

Tatsuoka et al.\ first derived a bound for the truncation
error~\cite[Proposition 3]{TatsuEn} and, based on the result,
proposed the following procedure for choosing $[l,r]$ so that
the relative truncation error in the $2$-norm is at most $\varepsilon>0$.

\begin{proposition}[Tatsuoka et al.~{\cite[Proposition 5]{TatsuEn}}]\label{prop:finite_interval}
Let $A \ne I$, and $\theta$ be a lower bound for $\|\log(A)\|_2$. Assume that the tolerance $\varepsilon > 0$ satisfies the following condition:
\begin{align*}
\varepsilon <
\frac{3\|A - I\|_2 \|A^{-1}\|_2}
{\theta(1 + \|A^{-1}\|_2)}.
\end{align*}
Let $s$ be the positive real solution to the following equation:
\begin{align}
\frac{1}{\theta}\left(-\log(s) + \frac{1 - s}{2s}\right)
\|A - I\|_2 \|A^{-1}\|_2 = \frac{\varepsilon}{2}.
\label{eq:finite_s}
\end{align}
Then, we define:
\begin{alignat}{2}
l &= \arsinh\left(\frac{2}{\pi}\artanh(2a - 1)\right),
&\qquad
r &= \arsinh\left(\frac{2}{\pi}\artanh(2b - 1)\right),
\label{eq:finite_lr}
\intertext{where}
a &= \min\left\{ \frac{\theta\varepsilon}{3\|A - I\|_2}, \frac{1}{2\|A - I\|_2} \right\},
&\qquad
b &= \max\left\{ s, \frac{2\|A^{-1}\|_2}{2\|A^{-1}\|_2 + 1} \right\}.
\label{eq:finite_ab}
\end{alignat}
Under these conditions, the following inequality holds:
\begin{align*}
\frac{\left\|\log(A) - \int_{l}^{r} F_{\mathrm{DE}}(x)\D x\right\|_2}{\|\log(A)\|_2}
\le \varepsilon.
\end{align*}
\end{proposition}

\begin{remark}\label{rem:revised-s}
Finding $s$ in Eq.~\eqref{eq:finite_s} is generally difficult.
Since $s \approx 1$ when the tolerance $\varepsilon$ is sufficiently small,
Tatsuoka et al.~\cite{TatsuEn} approximate its left-hand expression near $s=1$ using
$-\log s+(1-s)/(2s)\approx 3(1-s)/2$, giving
\[
\tilde{s}
= 1 - \frac{\theta \varepsilon}{3 \| A-I\|_2 \| A^{-1}\|_2},
\]
and use $\tilde{s}$ as a surrogate for $s$.
However, for $s<1$,
$-\log s+(1-s)/(2s)>3(1-s)/2$, so the first-order approximation
underestimates the exact expression. Consequently, replacing $s$ by
$\tilde{s}$ does not, in general, guarantee the key inequality
\begin{equation}
\frac{1}{\theta}\left(-\log(b) + \frac{1-b}{2b}\right)
\|A-I\|_2 \|A^{-1}\|_2\leq \frac{\varepsilon}{2},
\label{eq:finite_half}
\end{equation}
which is indispensable in their proof.
As a remedy, one can use the alternative upper bound
$-\log s+(1-s)/(2s)\le 2(1-s)$ for $2/3\le s\le1$. Solving the
corresponding equality gives the corrected surrogate
\begin{equation}
\tilde{s}
= 1 - \frac{\theta \varepsilon}{4\|A-I\|_{2}\|A^{-1}\|_{2}}.
\label{eq:revise-s}
\end{equation}
Since $b\ge\max\{\tilde{s},2\|A^{-1}\|_2/(2\|A^{-1}\|_2+1)\}\ge2/3$,
we have $-\log b+(1-b)/(2b)\le2(1-b)\le2(1-\tilde{s})$;
substituting~\eqref{eq:revise-s} gives~\eqref{eq:finite_half}.
\end{remark}

\subsection{Evaluation of the discretization error}\label{sec:4-2-3}

We now review Tatsuoka et al.'s estimate of the trapezoidal
discretization error in~\eqref{eq4.3}.
Ignoring the contribution outside $[l,r]$, this reduces to bounding
\[
\int_{-\infty}^{\infty} F_{\mathrm{DE}}(x)\D x
- h \sum_{k=-\infty}^{\infty} F_{\mathrm{DE}}(kh).
\]
When $A$ is Hermitian positive definite, unitary diagonalization reduces
the $2$-norm of the matrix discretization error to the maximum of the
corresponding scalar discretization errors over the eigenvalues of $A$.
More precisely,
\begin{equation}
\left\|
\int_{-\infty}^{\infty} F_{\mathrm{DE}}(x)\D x
- h \sum_{k=-\infty}^{\infty} F_{\mathrm{DE}}(kh)
\right\|_2
=
\max_{\lambda \in \Lambda(A)}
\left|
\int_{-\infty}^{\infty} f_{\mathrm{DE}}(x,\lambda) \D x
- h \sum_{k=-\infty}^{\infty} f_{\mathrm{DE}}(kh,\lambda)
\right|,
\label{eq:reduction-to-scalar}
\end{equation}
where $\Lambda(A)$ denotes the spectrum of $A$, and $f_{\mathrm{DE}}(x,\lambda)$
is defined by
\begin{equation}
f_{\mathrm{DE}}(x,\lambda)
= \frac{(\lambda - 1)\pi\cosh x}{(\lambda + \E^{-\pi\sinh x})(1 + \E^{\pi\sinh x})}.
\label{eq:scalar-fDE}
\end{equation}
This is the scalar counterpart of $F_{\mathrm{DE}}(x)$ in
Eq.~\eqref{eq:FDE-x}, obtained by replacing $A$ and $I$ with
$\lambda$ and $1$, respectively.
To analyze the discretization error,
the following function space is frequently used. Here, let $\domD_d$ denote the strip region defined by
$\domD_d=\{\zeta\in\mathbb{C}: |\Im\zeta|< d\}$ for $d>0$.

\begin{definition}
\label{def:func_space_BDd}
Let $d$ be a positive real number.
Then, $\mathbf{B}(\domD_d)$ denotes the family of all functions
$f$ that are analytic in $\domD_d$
and satisfy
\begin{align*}
\lim_{x \to \pm \infty} \int_{-d}^{d} |f(x + \I y)| \D y = 0
\end{align*}
and $\mathcal{N}(f,d)$ is finite, where
\begin{align*}
 \mathcal{N}(f,d)
 = \lim_{y\to d - 0}
\int_{-\infty}^{\infty}
\left\{ |f(x+\I y)|+|f(x-\I y)| \right\} \D x.
\end{align*}
\end{definition}

The discretization error for the integrand $f\in\mathbf{B}(\domD_d)$
has been estimated as follows.

\begin{theorem}[classical; e.g.~Stenger~{\cite[Theorem~3.4]{Stenger1981}}]
\label{theo1}
Let $f\in\mathbf{B}(\domD_d)$. Then,
\[
\left|
\int_{-\infty}^{\infty}f(x)\D x - h\sum_{k=-\infty}^{\infty}f(kh)
\right|
\leq\mathcal{N}(f,d)\frac{\E^{-2\pi d/h}}{1-\E^{-2\pi d/h}}.
\]
\end{theorem}

The strip of analyticity needed for this bound is given by the following
result of Tatsuoka et al.

\begin{proposition}[Tatsuoka et al.~{\cite[Proposition 2]{Tatsu}}]\label{prop:analytic_strip}
For $\lambda>0$, let
\begin{align}\label{eq:4-1}
d_0(\lambda)
= \arcsin\left( \sqrt{ \frac{(\log\lambda)^2 + 2\pi^2 - \sqrt{\left\{(\log\lambda)^2 + 2\pi^2\right\}^2 - 4\pi^4}} {2\pi^2} } \right).
\end{align}
Then, $f_{\mathrm{DE}}(z,\lambda)$ is analytic in the strip $\mathscr{D}_{d_0(\lambda)}$.
\end{proposition}

For fixed $\delta\in(0,d_0(\lambda))$, the integrand $f_{\mathrm{DE}}(z,\lambda)$ decays
double-exponentially on the smaller strip
$\domD_{d_0(\lambda)-\delta}$, so its boundary integral ($\mathcal{N}(f_{\mathrm{DE}}(\cdot,\lambda),d_0(\lambda)-\delta)$) is finite. Hence, for sufficiently small $h$, Theorem~\ref{theo1} gives
\begin{align}
\left|
\int_{-\infty}^{\infty} f_{\mathrm{DE}}(x, \lambda)\D x
- h \sum_{k=-\infty}^{\infty} f_{\mathrm{DE}}(kh, \lambda)
\right|
= \OO\left( \exp\left(-\frac{2\pi (d_0(\lambda)-\delta)}{h}\right) \right).
\label{eq:4-23}
\end{align}

\begin{remark}\label{rem:strip_margin}
Tatsuoka et al.~\cite[Eq.~(6)]{Tatsu} state the sharper rate
$\OO(\E^{-2\pi d_0(\lambda)/h})$. The argument above establishes only
$\OO(\E^{-2\pi(d_0(\lambda)-\delta)/h})$ for each fixed $\delta>0$:
it gives no uniform bound on $\mathcal{N}(f_{\mathrm{DE}}(\cdot,\lambda),d_0(\lambda)-\delta)$ as $\delta\to0$. Analyticity
in $\domD_{d_0(\lambda)}$ alone does not ensure that the boundary
integral there is finite, so the sharper rate requires a separate estimate.

%
To derive the subsequent matrix error estimates, fix $q\in(0,1)$ independently
of $A$ and $h$. Setting $\delta=(1-q)d_0(\lambda)$
in~\eqref{eq:4-23} yields
\[
\left|
\int_{-\infty}^{\infty}f_{\mathrm{DE}}(x,\lambda)\D x
-h\sum_{k=-\infty}^{\infty}f_{\mathrm{DE}}(kh,\lambda)
\right|
=\OO\left(\exp\left(-\frac{2\pi q d_0(\lambda)}{h}\right)\right).
\]
\end{remark}

\begin{remark}\label{rem:d0_simpler}
Note that $d_0(\lambda)$ defined by~\eqref{eq:4-1} can be rewritten in the following simpler form:
\begin{align}
d_0(\lambda)
= \begin{cases}
\arccos\left(
\sqrt{
\dfrac{2}{1 + \sqrt{1 + \left(2\pi/\log\lambda\right)^2}}
}
\right) & (\lambda \neq 1),\\
\dfrac{\pi}{2} & (\lambda = 1).
\end{cases}
\label{eq:d0-simple}
\end{align}
\end{remark}

\subsection{Overall error estimate and the effect of the condition number on convergence}
\label{subsec:total_error}

We review Tatsuoka et al.'s convergence analysis~\cite{Tatsu}, retaining
the strip-width qualification in Remark~\ref{rem:strip_margin}. By
Eq.~\eqref{eq:d0-simple}, $d_0(\lambda)=d_0(1/\lambda)$ and decreases
as $|\log\lambda|$ increases; thus eigenvalues farther from $1$ give
slower exponential decay.

By Eq.~\eqref{eq:reduction-to-scalar}, the matrix discretization error
is the maximum scalar error over the eigenvalues. Positive scalar
scaling can improve this rate without changing the condition number:
let $\widehat{A}$ be given, set $A=p\widehat{A}$ for $p>0$, and denote
the extreme eigenvalues of $\widehat{A}$ and $A$ by
$\widehat{\lambda}_{\min},\widehat{\lambda}_{\max}$ and
$\lambda_{\min}=p\widehat{\lambda}_{\min},\lambda_{\max}=p\widehat{\lambda}_{\max}$,
respectively. Then $\kappa(A)=\kappa(\widehat{A})$, where $\kappa(B)=\|B\|_2\|B^{-1}\|_2$
denotes the spectral condition number of a nonsingular matrix $B$.
The monotonicity of $d_0$ gives
\[
\min_{\lambda\in\Lambda(A)}d_0(\lambda)
=\min\{d_0(\lambda_{\min}),d_0(\lambda_{\max})\}.
\]
Since $d_0$ decreases with $|\log\lambda|$, maximizing this minimum
over $p$ is equivalent to minimizing
\[
\max\{|\log(p\widehat{\lambda}_{\min})|,
|\log(p\widehat{\lambda}_{\max})|\}.
\]
The minimum occurs when $\lambda_{\min}\lambda_{\max}=1$, i.e.,
$p=1/\sqrt{\widehat{\lambda}_{\min}\widehat{\lambda}_{\max}}$.
Henceforth $A=p\widehat{A}$ denotes this normalization; it equalizes the
endpoint strip widths and maximizes the minimum over $\lambda\in\Lambda(A)$.
To preserve Remark~\ref{rem:strip_margin}, fix $q\in(0,1)$ independently
of $A,h$ and use the common strip $\domD_{q d_0(\lambda_{\max})}$, which
lies strictly inside every eigenvalue's analyticity strip. Applying
Eq.~\eqref{eq:4-23} to each $\lambda\in\Lambda(A)$ gives, for fixed
$A$ and $q$,
\begin{equation}
\max_{\lambda \in \Lambda(A)}
\left|
\int_{-\infty}^{\infty} f_{\mathrm{DE}}(x,\lambda) \D x
- h \sum_{k=-\infty}^{\infty} f_{\mathrm{DE}}(kh,\lambda)
\right|
=
\OO\left(\exp\left(-\frac{2\pi q d_0(\lambda_{\max})}{h}\right)\right)
=
\OO\left(\exp\left(-\frac{2\pi q d_0(\sqrt{\kappa(A)})}{h}\right)\right),
\nonumber
\end{equation}
as $h\to0$.
Apply the quadrature to $A=p\widehat{A}$ and recover
$\log(\widehat{A})=\log(A)-(\log p)I$.

Furthermore, let $A\ne I$ be Hermitian positive definite with
$\lambda_{\min}\lambda_{\max}=1$, and fix a finite interval $[l,r]$
such that the norm of the sum of the two tail integrals is
at most $\theta\varepsilon$. This condition is ensured by
Proposition~\ref{prop:finite_interval}; it is also retained by the
revised endpoint construction in Remark~\ref{rem:revised-s}.
With $M\ge2$ quadrature nodes and $h=(r-l)/(M-1)$, the following
estimate holds for sufficiently small $h$:
\begin{align*}
&\left\|
\log(A)
-
h \sum_{k=1}^{M-2} F_{\mathrm{DE}}(l+kh)
-\frac{h}{2}\left\{F_{\mathrm{DE}}(l)+F_{\mathrm{DE}}(r)\right\}
\right\|_2\\
&\leq
\left\|
\int_{-\infty}^{\infty} F_{\mathrm{DE}}(x)\D x
- h \sum_{k=-\infty}^{\infty} F_{\mathrm{DE}}(l+kh)
\right\|_2
+
\left\|
h \sum_{k=-\infty}^{-1} F_{\mathrm{DE}}(l+kh)
+h \sum_{k=1}^{\infty} F_{\mathrm{DE}}(r+kh)
+\frac{h}{2}\left\{F_{\mathrm{DE}}(l)+F_{\mathrm{DE}}(r)\right\}
\right\|_2\\
&\leq
\left\|
\int_{-\infty}^{\infty} F_{\mathrm{DE}}(l+x)\D x
- h \sum_{k=-\infty}^{\infty} F_{\mathrm{DE}}(l+kh)
\right\|_2
+2
\left\|
\int_{-\infty}^{l}F_{\mathrm{DE}}(x)\D x
+\int_{r}^{\infty}F_{\mathrm{DE}}(x)\D x
\right\|_2\\
&\leq\OO\left(
\exp\left(-\frac{2\pi q d_0(\sqrt{\kappa(A)})}{r-l}(M-1)\right)
\right)
+2\theta\varepsilon.
\end{align*}
In the second inequality, the estimate $\|T_h\|_2\le2\|T\|_2$
from Lemma~\ref{lem:tail_trapezoidal_bound} is used with $0<h<h_0$.
The discretization estimate also applies to the shifted grid $l+kh$,
because a real translation preserves the analytic strip and the boundary
integral in Theorem~\ref{theo1}.
The parameter $q$ is fixed as above.

Consequently, incorporating the qualifications discussed above, the overall
error bound for the method of Tatsuoka et al.\ can be stated as follows.
\begin{theorem}
\label{thm:Tatsuoka-et-al}
Let $A \ne I$ be a Hermitian positive definite matrix, with maximum eigenvalue $\lambda_{\max}$ and minimum eigenvalue $\lambda_{\min}$ satisfying
$\lambda_{\max}\lambda_{\min}=1$.
Choose a positive lower bound $\theta\in(0,\|\log(A)\|_2]$ and a truncation
error tolerance $\varepsilon\in(0,3\|A-I\|_2/(2\theta))$,
and define $\tilde{s}$ by~\eqref{eq:revise-s}.
Define $a$ and $b$ by~\eqref{eq:finite_ab}, with $s$ replaced by
$\tilde{s}$ in the definition of $b$, and define $l$ and $r$
by~\eqref{eq:finite_lr}.
Let $d_0(\lambda)$ be defined by~\eqref{eq:d0-simple}.
For each integer $M\ge2$, let $h=(r-l)/(M-1)$.
Fix $q\in(0,1)$, independently of $A$ and $M$.
Then there exist constants $C>0$ and $M_0\ge2$, independent of $M$,
such that the following inequality holds for all integers $M\ge M_0$:
\begin{align*}
\left\|\log(A) - h\left\{\frac{1}{2}F_{\mathrm{DE}}(l) + \sum_{k=1}^{M-2} F_{\mathrm{DE}}(l+kh) + \frac{1}{2}F_{\mathrm{DE}}(r)\right\}\right\|_{2}
\leq
C
\exp\left(-\frac{2\pi q d_0(\sqrt{\kappa(A)})(M-1)}{r-l}\right)
+2\theta\varepsilon.
\end{align*}
\end{theorem}

For normalized Hermitian positive definite matrices as above,
Tatsuoka et al.~\cite{Tatsu} compare the theoretical convergence rates
of the DE formula and GL quadrature in terms of the number of quadrature nodes.
They report that GL quadrature converges faster when
$\kappa(A)\lesssim 2.7\times10^3$, whereas the DE formula converges faster
when $\kappa(A)\gtrsim 2.7\times10^3$.

\subsection{Adaptive implementation of the DE formula with reuse of previous evaluations}
\label{sec:adaptive-Tatsuoka}

Tatsuoka et al.~\cite[Algorithm~2]{TatsuEn} halve the mesh on a fixed
interval, reuse previous evaluations, and stop according to successive
approximations. Thus, adaptivity selects the number of global refinement
steps rather than locally refining parts of the interval. For consistency
with the transformation convention adopted here, we express their update
formula and stopping criterion using $t=1/(1+\E^{-\pi\sinh x})$.

To retain the placement of the factor $A-I$ in their stopping criterion,
introduce the auxiliary integrand
\[
G_{\mathrm{DE}}(x)
=\frac{\pi\cosh x}{2\cosh^2((\pi/2)\sinh x)}
\left[\left\{1+\tanh\left(\frac{\pi}{2}\sinh x\right)\right\}(A-I)
+2I\right]^{-1}.
\]
Thus, $F_{\mathrm{DE}}(x)=(A-I)G_{\mathrm{DE}}(x)$ and
$\log(A)=(A-I)\int_{-\infty}^{\infty}G_{\mathrm{DE}}(x)\D x$.
Fix the interval $[l,r]$ using the truncation tolerance $\varepsilon$.
For a positive integer $n=M-1$, define
\begin{equation}
\widehat S_n
= h\left\{
\frac{1}{2}G_{\mathrm{DE}}(l)+\frac{1}{2}G_{\mathrm{DE}}(r)
+ \sum_{k=1}^{n-1} G_{\mathrm{DE}}(l+kh)
\right\},
\qquad h=\frac{r-l}{n}.
\label{eq:Tatsuoka-trapezoidal-sum}
\end{equation}
When $n$ is doubled, the mesh size becomes $h/2$.
The even-indexed nodes of the refined grid coincide with the nodes of
the original grid, while the odd-indexed nodes are the new midpoints.
Separating these contributions gives
\begin{equation}
\widehat S_{2n}
=
\frac{1}{2}\widehat S_n
+
\frac{h}{2}\sum_{j=1}^{n} G_{\mathrm{DE}}\left(l+(2j-1)\frac{h}{2}\right).
\label{eq:adaptive_update}
\end{equation}
Thus, only $n$ new midpoint evaluations are needed; the node count grows
from $M=n+1$ to $2M-1=2n+1$.

For the stopping criterion, Tatsuoka et al.~\cite[Eq.~(22)]{TatsuEn}
use an approximate error estimate for the trapezoidal rule based on
successive mesh sizes. Applied to $G_{\mathrm{DE}}$, this estimate takes
the form
\[
\left\|\int_l^r G_{\mathrm{DE}}(x)\D x-\widehat S_{2n}\right\|_2
\approx\frac13\|\widehat S_{2n}-\widehat S_n\|_2
\]
for sufficiently small $h$ when the error of the trapezoidal rule is
approximately proportional to $h^2$; halving $h$ then reduces the error
by about a factor of four, yielding the factor $1/3$.
Following their stopping criterion, define the normalized error indicator
\begin{equation*}
\mathrm{err}_n:=\frac{\|\widehat S_{2n}-\widehat S_n\|_2}{3\theta},
\qquad 0<\theta\le\|\log(A)\|_2.
\end{equation*}
In particular, the factor $A-I$ is not included in this indicator;
it is applied to the quadrature result only when forming the output.
Let $\tau>0$ be the tolerance for this indicator, distinct from the
interval truncation tolerance $\varepsilon$.
Algorithm~\ref{alg:adaptive-Tatsuoka} includes the preprocessing
described by Tatsuoka et al.~\cite[Algorithm~2]{TatsuEn} to compute
$\theta$, $l$, and $r$.
Here, $\rho(A)$ denotes the spectral radius of $A$, and the initial
number of subintervals $n_0$ corresponds to $m_0-1$ in their notation.
As in Theorem~\ref{thm:Tatsuoka-et-al}, $s$ in~\eqref{eq:finite_ab}
is replaced by $\tilde{s}$ defined in~\eqref{eq:revise-s}.
When applying Theorem~\ref{thm:Tatsuoka-et-al}, its
assumptions on $A$ and the additional restriction
$\varepsilon<3\|A-I\|_2/(2\theta)$ must also be satisfied.
The choice $\theta=|\log(\rho(A))|$ requires $\rho(A)\ne1$ for the
normalization to be defined; if $\rho(A)=1$, a different positive
lower bound for $\|\log(A)\|_2$ is needed when $A\ne I$, while
$A=I$ can be handled by returning the zero matrix.

\begin{algorithm}[tbp]
\caption{Adaptive algorithm of Tatsuoka et al.\ based on the DE formula with a fixed truncation interval}
\label{alg:adaptive-Tatsuoka}
\begin{algorithmic}[1]
\Require Matrix $A$ with no eigenvalues on $(-\infty,0]$ and $\rho(A)\ne1$.
\Require Positive integers $n_0$ and $M_{\max}\geq n_0+1$;
truncation tolerance $\varepsilon>0$; indicator tolerance $\tau>0$.
\Ensure Approximation to $\log(A)$ and a termination status.
\State Compute $\mu\gets\|A-I\|_2$, $\nu\gets\|A^{-1}\|_2$, and $\rho(A)$.
\State $\theta\gets|\log(\rho(A))|$, $\varepsilon_{\max}\gets3\mu\nu/(\theta(1+\nu))$
\If{$\varepsilon\geq\varepsilon_{\max}$}
  \State $\varepsilon\gets\varepsilon_{\max}/2$
\EndIf
\State $a\gets\min\{\theta\varepsilon/(3\mu),1/(2\mu)\}$
\State $\tilde{s}\gets1-\theta\varepsilon/(4\mu\nu)$
\State $b\gets\max\{\tilde{s},2\nu/(2\nu+1)\}$
\State $l\gets\arsinh((2/\pi)\artanh(2a-1))$
\State $r\gets\arsinh((2/\pi)\artanh(2b-1))$
\State $n\gets n_0$, $h\gets(r-l)/n$
\State Compute $\widehat S_n$ using~\eqref{eq:Tatsuoka-trapezoidal-sum}.
\While{$2n+1\leq M_{\max}$}
  \State Compute $\widehat S_{2n}$ using~\eqref{eq:adaptive_update}.
  \State $\mathrm{err}_n\gets\|\widehat S_{2n}-\widehat S_n\|_2/(3\theta)$
  \If{$\mathrm{err}_n\leq\tau$}
    \State \Return $((A-I)\widehat S_{2n},\text{indicator satisfied})$
  \EndIf
  \State $\widehat S_n\gets\widehat S_{2n}$, $n\gets2n$, $h\gets h/2$
\EndWhile
\State \Return $((A-I)\widehat S_n,\text{refinement-limit warning})$
\end{algorithmic}
\end{algorithm}
This stopping test uses an approximate error indicator, not a rigorous
upper bound, and therefore does not by itself guarantee the requested
accuracy. It estimates the quadrature error for the integral of
$G_{\mathrm{DE}}$, before multiplication by $A-I$, rather than the
error in the resulting approximation to $\log(A)$ directly.
Moreover, the difference $\widehat S_{2n}-\widehat S_n$ does not measure the
interval truncation error, since both approximations use the same fixed
interval $[l,r]$. That error must be controlled separately through the
choice of $\varepsilon$ and the endpoints.

\section{Algorithms for computing the matrix logarithm proposed in this study}
\label{sec:our-algorithm}

\subsection{Idea of the approximation}
In this study, we consider computing the integral shown in~\eqref{eq:logA-after-DEtrans} via the following two-step approximation:
\begin{align*}
\int_{-\infty}^{\infty} F_{\mathrm{DE}}(x)\D x
\approx h \sum_{k=-\infty}^{\infty} F_{\mathrm{DE}}(kh)
\approx h \sum_{k=-N}^{N} F_{\mathrm{DE}}(kh).
\end{align*}
The first approximation applies the composite trapezoidal rule to the
integral over the infinite interval $(-\infty,\infty)$.
The second approximation truncates the infinite sum.
In the same manner as Tatsuoka et al., we first investigate the
truncation error in Section~\ref{sec:evaluate-truncation}, followed by
the discretization error in Section~\ref{sec:evaluate-discretize}.
Combining these results, we state the overall error estimate in
Section~\ref{sec:evaluate-total} and extend it to diagonalizable matrices
with positive real eigenvalues in Section~\ref{sec:evaluate-total-diagonalizable}.
Finally, Section~\ref{sec:adaptive-new} presents practical algorithms
based on the preceding analysis.

\subsection{Evaluation of the truncation error of the infinite trapezoidal sum}
\label{sec:evaluate-truncation}
When $A$ is Hermitian positive definite,
by the same argument used to derive~\eqref{eq:reduction-to-scalar},
the truncation error satisfies
\begin{align}
\left\|
h\sum_{k=-\infty}^{-N-1} F_{\mathrm{DE}}(kh)
+ h\sum_{k=N+1}^{\infty}   F_{\mathrm{DE}}(kh)
\right\|_2
=
\max_{\lambda \in \Lambda(A)}
\left|
h\sum_{k=-\infty}^{-N-1} f_{\mathrm{DE}}(kh,\lambda)
+ h\sum_{k=N+1}^{\infty} f_{\mathrm{DE}}(kh,\lambda)
\right|,
\label{eq:4-24}
\end{align}
where $f_{\mathrm{DE}}(x,\lambda)$ is defined by~\eqref{eq:scalar-fDE}.
Using this reduction, we obtain the following result.
\begin{lemma}\label{lemma:FDE_tail_bound}
Let $A$ be a Hermitian positive definite matrix,
let $h>0$, and let $N$ be a nonnegative integer.
Then,
\begin{align*}
\left\| h\sum_{k=-\infty}^{-N-1} F_{\mathrm{DE}}(kh)
+ h\sum_{k=N+1}^{\infty} F_{\mathrm{DE}}(kh) \right\|_2
\le \left\| A - A^{-1} \right\|_2  \E^{-\pi \sinh(Nh)}.
\end{align*}
\end{lemma}

\begin{proof}
For the negative tail on the right-hand side of Eq.~\eqref{eq:4-24}, we have
\begin{align*}
  \left| h\sum_{k=-\infty}^{-N-1} f_{\mathrm{DE}}(kh,\lambda) \right|
  &\le h\sum_{k=-\infty}^{-N-1}
       \frac{|\lambda-1|}{\lambda + \E^{-\pi\sinh(kh)}} \cdot \frac{\pi\cosh(kh)}{1 + \E^{\pi\sinh(kh)}} \\
  &= h\sum_{k=-\infty}^{-N-1}
       \frac{|\lambda-1|}{\lambda \E^{\pi\sinh(kh)} + 1} \cdot \frac{\pi\cosh(kh)}{1 + \E^{\pi\sinh(kh)}} \E^{\pi\sinh(kh)} \\
  &\le |\lambda-1|h\sum_{k=-\infty}^{-N-1}
       \pi\cosh(kh)\E^{\pi\sinh(kh)} \\
  &\le |\lambda-1| \int_{-\infty}^{-Nh} (\pi\cosh x) \E^{\pi\sinh x} \D x \\
  &= |\lambda-1| \E^{-\pi\sinh(Nh)}.
\end{align*}
The last inequality uses the fact that
$(\pi\cosh x)\E^{\pi\sinh x}$ is increasing on $\mathbb{R}$.
Similarly, for the positive tail, we have
\begin{align*}
  \left| h\sum_{k=N+1}^{\infty} f_{\mathrm{DE}}(kh,\lambda) \right|
  &\le h\sum_{k=N+1}^{\infty}
       \frac{|\lambda-1|}{\lambda + \E^{-\pi\sinh(kh)}} \cdot \frac{\pi\cosh(kh)}{1 + \E^{\pi\sinh(kh)}} \\
  &= h\sum_{k=N+1}^{\infty}
       \frac{|\lambda-1|}{\lambda + \E^{-\pi\sinh(kh)}} \cdot \frac{\pi\cosh(kh)}{ \E^{-\pi\sinh(kh)} + 1} \E^{-\pi\sinh(kh)} \\
  &\le \frac{|\lambda-1|}{\lambda}h\sum_{k=N+1}^{\infty}
       \pi\cosh(kh)\E^{-\pi\sinh(kh)} \\
  &\le  \frac{|\lambda-1|}{\lambda} \int_{Nh}^{\infty} (\pi\cosh x) \E^{-\pi\sinh x} \D x \\
  &=  \frac{|\lambda-1|}{\lambda} \E^{-\pi\sinh(Nh)}.
\end{align*}
Here, the last inequality uses the fact that
$(\pi\cosh x)\E^{-\pi\sinh x}$ is decreasing on $\mathbb{R}$.
Therefore,
\begin{align*}
\left\| h\sum_{k=-\infty}^{-N-1} F_{\mathrm{DE}}(kh)
      + h\sum_{k=N+1}^{\infty} F_{\mathrm{DE}}(kh) \right\|_2
&\le \max_{\lambda \in \Lambda(A)}
      \left\{ |\lambda - 1| + \frac{|\lambda - 1|}{\lambda} \right\}
      \E^{-\pi \sinh(Nh)} \\
&= \max_{\lambda \in \Lambda(A)}
      \frac{|\lambda - 1|(\lambda + 1)}{\lambda}
      \E^{-\pi \sinh(Nh)} \\
&= \max_{\lambda \in \Lambda(A)}
      \left| \frac{(\lambda - 1)(\lambda + 1)}{\lambda} \right|
      \E^{-\pi \sinh(Nh)} \\
&= \max_{\lambda \in \Lambda(A)}
      \left| \lambda - \frac{1}{\lambda} \right|
      \E^{-\pi \sinh(Nh)},
\end{align*}
which completes the proof of the lemma.
\end{proof}

\subsection{Evaluation of the discretization error}
\label{sec:evaluate-discretize}

To evaluate the discretization error,
the following lemma is essential.
The proof is deferred to Section~\ref{sec:proof-essential-lemma}.

\begin{lemma}
\label{lem:bound-None}
Let $\lambda$ be a positive real number,
let $L = |\log\lambda|$, and
let $d_0(\lambda)$ be defined by~\eqref{eq:d0-simple}.
For each $y\in(0,d_0(\lambda))$, define
\begin{align}
 x_{L,y}&=\arsinh\left(\frac{L}{\pi\cos y}\right),
 \label{eq:x-L-y}\\
 T_{L,y}&=\frac{\pi\cos^2y-L\sin y}{2\sin y},
 \label{eq:T-L-y}\\
 \tau_{L,y}&=\frac{T_{L,y}}{1+\sqrt{T_{L,y}}},
 \label{eq:tau-L-y}\\
 a_{L,y}&=\frac{\pi}{2}\sin y
 \sqrt{1+\left(
 \sinh x_{L,y}+\frac{2\tau_{L,y}}{\pi\cos y}
 \right)^2}.
 \label{eq:a-L-y}
\end{align}
Then, for every $d\in(0,d_0(\lambda))$,
$f_{\mathrm{DE}}(\cdot,\lambda)\in\mathbf{B}(\domD_d)$.
In particular,
$\mathcal{N}(f_{\mathrm{DE}}(\cdot,\lambda),d)$ is bounded by
\begin{equation*}
\mathcal{N}(f_{\mathrm{DE}}(\cdot,\lambda),d)
\leq \frac{2L}{\gamma_d},
\end{equation*}
where $\gamma_y$ is defined by
\begin{equation}
\gamma_y = \cos\left(\frac{\pi}{2}\sin y\right)
\cos(a_{L,y})\cos y.
\label{eq:gamma_y}
\end{equation}
\end{lemma}

Taking this lemma for granted for now, we combine it with
Theorem~\ref{theo1} to obtain the following discretization-error bound.

\begin{theorem}
\label{thm:our-discretization-error}
Let $A$ be a Hermitian positive definite matrix,
let $L = \max_{\lambda\in\Lambda(A)}|\log\lambda|$, and
let $d_0(\lambda)$ be defined by~\eqref{eq:d0-simple}.
Define
\begin{equation*}
d_*(A)=\min_{\lambda\in\Lambda(A)}d_0(\lambda).
\end{equation*}
For each $y\in(0,d_*(A))$,
let $x_{L,y}$, $T_{L,y}$, $\tau_{L,y}$, and
$a_{L,y}$ be defined
by~\eqref{eq:x-L-y}--\eqref{eq:a-L-y}, respectively.
Fix $q\in(0,1)$, independently of $A$ and $h$.
Then, for any $h>0$, we have
\[
\left\|
\log(A) - h\sum_{k=-\infty}^{\infty}F_{\mathrm{DE}}(kh)
\right\|_2
\leq
\frac{2\|\log(A)\|_2}{\gamma_{q d_*(A)}}
\cdot
\frac{\E^{-2\pi q d_*(A)/h}}
{1-\E^{-2\pi q d_*(A)/h}},
\]
where $\gamma_y$ is defined by~\eqref{eq:gamma_y}.
\end{theorem}
\begin{proof}
If $A=I$, the result is immediate.  Otherwise, since $d_0(\lambda)$
depends on $\lambda$ only through $|\log\lambda|$ and decreases as
$|\log\lambda|$ increases, we have
\[
 d_*(A)=d_0(\E^L).
\]
For each $\lambda\in\Lambda(A)$, put
$L_\lambda=|\log\lambda|$ and denote by $\gamma_{\lambda,y}$ the
quantity defined by~\eqref{eq:gamma_y} with $L=L_\lambda$.
Applying Lemma~\ref{lem:bound-None} with $d=q d_*(A)$ and then
Theorem~\ref{theo1}, we obtain
\[
 \left|
 \log\lambda-h\sum_{k=-\infty}^{\infty}
 f_{\mathrm{DE}}(kh,\lambda)
 \right|
 \leq
 \frac{2L_\lambda}{\gamma_{\lambda,q d_*(A)}}
 \frac{\E^{-2\pi q d_*(A)/h}}
 {1-\E^{-2\pi q d_*(A)/h}}.
\]
This estimate depends on the individual eigenvalue $\lambda$.
We next bound all these estimates uniformly by the one corresponding
to the worst-case value $L$.

For fixed $y\in(0,d_*(A))$, regard the quantities in
Eqs.~\eqref{eq:x-L-y}--\eqref{eq:a-L-y} as functions of $L$, and put
$r=\sqrt{T_{L,y}}$.  Squaring this relation and rearranging the
definition of $T_{L,y}$ gives
\[
 L=\frac{\pi\cos^2y}{\sin y}-2r^2.
\]
Using this identity, we may write
\[
 \sinh x_{L,y}+\frac{2\tau_{L,y}}{\pi\cos y}
 =\frac{L}{\pi\cos y}
 +\frac{2r^2}{\pi\cos y (1 + r)}
 =\frac{\cos y}{\sin y}
 -\frac{2r^3}{\pi\cos y(1+r)}.
\]
For fixed $y$, since $T_{L,y}$ decreases with $L$,
$r=\sqrt{T_{L,y}}$ decreases with $L$.  Moreover,
$r^3/(1+r)$ increases with $r>0$, and hence the expression
on the left increases with $L$.  Since $0<a_{L,y}<\pi/2$ for
$0<y<d_L$, it follows that $a_{L,y}$ increases and $\gamma_y$
decreases with $L$.  Since $L_\lambda\leq L$, we obtain
that
\[
 \gamma_{\lambda,q d_*(A)}
 \geq\gamma_{q d_*(A)},
 \qquad
 \frac{L_\lambda}{\gamma_{\lambda,q d_*(A)}}
 \leq\frac{L}{\gamma_{q d_*(A)}}.
\]
Taking the maximum over the eigenvalues of $A$ and using
$L=\|\log(A)\|_2$ proves the result.
\end{proof}

\begin{remark}
\label{rem:L-from-extreme-eigenvalues}
For a Hermitian positive definite matrix $A$, let $\lambda_{\min}$
and $\lambda_{\max}$ denote its minimum and maximum eigenvalues,
respectively. Then,
\[
 L=\|\log(A)\|_2
 =\max_{\lambda\in\Lambda(A)}|\log\lambda|
 =\max\{|\log\lambda_{\min}|,|\log\lambda_{\max}|\}
 =\max\{\log(1/\lambda_{\min}),\log\lambda_{\max}\}.
\]
Thus, evaluating $L$ requires only the two extreme eigenvalues of $A$,
not the computation of $\log(A)$ itself. Moreover, as shown in the proof
of Theorem~\ref{thm:our-discretization-error}, $d_*(A)=d_0(\E^L)$,
so the constants in the discretization-error bound can be evaluated
from these eigenvalues without forming $\log(A)$.
\end{remark}

\subsection{Evaluation of the overall error}
\label{sec:evaluate-total}
Using Theorem~\ref{theo1} and Lemma~\ref{lemma:FDE_tail_bound}, the error evaluation for the computation method in this study is as follows.

\begin{theorem}\label{thm:DE_logA_error_simple}
Let the assumptions of
Theorem~\ref{thm:our-discretization-error} be fulfilled.
Write $L=\|\log(A)\|_2$.
Furthermore, let $n$ be a positive integer and define
\begin{align}
h &= \frac{1}{n}, \label{eq:h_def} \\
N &= \left\lceil \frac{1}{h} \arsinh(\pi n) \right\rceil. \label{eq:N_def}
\end{align}
Then, we have
\[
\left\|\log(A)-h\sum_{k=-N}^{N}F_{\mathrm{DE}}(kh)\right\|_2
\leq
2\|\log(A)\|_2\left\{
\frac{1}{\gamma_{q d_*(A)}}
+\E^{2q\cos(d_*(A))}
\right\}
\frac{\E^{-2\pi q d_*(A)n}}{1-\E^{-2\pi q d_*(A)n}},
\]
where $\gamma_y$ is defined by~\eqref{eq:gamma_y}.
\end{theorem}

\begin{proof}
If $L=0$, then $A=I$, and the result
is immediate.  We therefore assume $L>0$.
By the proof of Theorem~\ref{thm:our-discretization-error},
$d_*(A)=d_0(\E^L)$.  Put $d=d_*(A)$ for short.
Then $0<d<\pi/2$, and Eq.~\eqref{eq:d0-simple} yields
\[
\left(\frac{2\pi}{L}\right)^2
=\left(\frac{2}{\cos^2d}-1\right)^2-1
=\frac{4\sin^2d}{\cos^4d}.
\]
Consequently, since $\tan d>d$, we have
\[
L=\frac{\pi\cos^2d}{\sin d},
\qquad
Ld=\pi\cos d\frac{d}{\tan d}<\pi\cos d.
\]
Put $\alpha=2\pi qd$.
Since $0<d<\pi/2$ and $q\in(0,1)$, we have $0<\alpha<\pi^2$, and
\begin{equation}
\frac{\alpha L}{\pi^2}
=2q \frac{L d}{\pi}<2q \cos d.
\label{eq:general-alpha-L-bound}
\end{equation}
Let $E_{\mathrm{disc}}$ and $E_{\mathrm{tail}}$ denote the errors
\begin{align*}
E_{\mathrm{disc}}
&=\left\|\log(A)-h\sum_{k=-\infty}^{\infty}
F_{\mathrm{DE}}(kh)\right\|_2,\\
E_{\mathrm{tail}}
&=\left\|h\sum_{k=-\infty}^{-N-1}F_{\mathrm{DE}}(kh)
+h\sum_{k=N+1}^{\infty}F_{\mathrm{DE}}(kh)\right\|_2.
\end{align*}
By Theorem~\ref{thm:our-discretization-error},
Lemma~\ref{lemma:FDE_tail_bound}, and the definitions of $h$ and $N$,
\[
E_{\mathrm{disc}}\leq \frac{2}{\gamma_{q d}} L
\frac{\E^{-\alpha n}}{1-\E^{-\alpha n}},
\qquad
E_{\mathrm{tail}}\leq 2\sinh L\E^{-\pi^2n}.
\]
Here, the truncation-error bound follows from
$\pi\sinh(Nh)\geq\pi^2n$ and
\[
\|A-A^{-1}\|_2
=\max_{\lambda\in\Lambda(A)}|\lambda-\lambda^{-1}|
=2\max_{\lambda\in\Lambda(A)}\sinh(|\log\lambda|)
\leq2\sinh L.
\]
In what follows, we consider two cases.

First, suppose that $\pi^2n\geq L$.
Since $\sinh L\leq L\E^L$ and $\pi^2-\alpha>0$,
\begin{align*}
E_{\mathrm{tail}}
&\leq 2L\E^{L-\pi^2n}
=2L\E^{L-(\pi^2-\alpha)n}\E^{-\alpha n}\\
&\leq 2L\E^{L-(\pi^2-\alpha)L/\pi^2}\E^{-\alpha n}
=2L\E^{\alpha L/\pi^2}\E^{-\alpha n}\\
&\leq 2\E^{2q\cos d}L\E^{-\alpha n},
\end{align*}
where the last inequality follows from~\eqref{eq:general-alpha-L-bound}.
Since $0<1-\E^{-\alpha n}<1$, we obtain
\[
E_{\mathrm{tail}}\leq 2\E^{2q\cos d}L
\frac{\E^{-\alpha n}}{1-\E^{-\alpha n}}.
\]
Adding the discretization-error bound proves the claimed estimate
in this case.

Next, suppose that $\pi^2n<L$.
For each $\lambda\in\Lambda(A)$, put
\[
Q_N(\lambda)=h\sum_{k=-N}^{N}f_{\mathrm{DE}}(kh,\lambda),
\qquad
Q_{\infty}(\lambda)=h\sum_{k=-\infty}^{\infty}
f_{\mathrm{DE}}(kh,\lambda).
\]
If $\lambda\ne1$, then $f_{\mathrm{DE}}(x,\lambda)$ has the
constant sign of $\lambda-1$ on the real axis.
Thus, $Q_N(\lambda)$, $Q_{\infty}(\lambda)$, and $\log\lambda$
have the same sign, and $|Q_N(\lambda)|\leq|Q_{\infty}(\lambda)|$.
Therefore, when $|\log\lambda|\leq|Q_N(\lambda)|$, it follows that
\[
 |\log\lambda-Q_N(\lambda)|
=|Q_N(\lambda)|-|\log\lambda|
\leq|Q_{\infty}(\lambda)|-|\log\lambda|
=|\log\lambda-Q_{\infty}(\lambda)|.
\]
When $|\log\lambda|\geq|Q_N(\lambda)|$, it follows that
\[
 |\log\lambda-Q_N(\lambda)|
=|\log\lambda|-|Q_N(\lambda)|
\leq|\log\lambda|.
\]
Combining these two cases gives
\[
|\log\lambda-Q_N(\lambda)|
\leq\max\left\{
|\log\lambda-Q_{\infty}(\lambda)|,
|\log\lambda|
\right\}.
\]
This inequality also holds for $\lambda=1$, since all terms vanish.
Taking the maximum over the eigenvalues of $A$ gives
\[
\left\|\log(A)-h\sum_{k=-N}^{N}F_{\mathrm{DE}}(kh)\right\|_2
\leq\max\{E_{\mathrm{disc}}, L\}.
\]
Moreover, since
$\alpha n<\alpha L/\pi^2<2q\cos d$, we have
\[
L=(\E^{\alpha n}-1)L
\frac{\E^{-\alpha n}}{1-\E^{-\alpha n}}
\leq(\E^{2q\cos d}-1)L
\frac{\E^{-\alpha n}}{1-\E^{-\alpha n}}
<2\E^{2q\cos d}L
\frac{\E^{-\alpha n}}{1-\E^{-\alpha n}}.
\]
Therefore,
\begin{align*}
\left\|\log(A)-h\sum_{k=-N}^{N}F_{\mathrm{DE}}(kh)\right\|_2
&\leq\max\left\{\frac{2}{\gamma_{q d}},2\E^{2q\cos d}\right\}L
\frac{\E^{-\alpha n}}{1-\E^{-\alpha n}}
\leq \left(\frac{2}{\gamma_{q d}}+2\E^{2q\cos d}\right)L
\frac{\E^{-\alpha n}}{1-\E^{-\alpha n}}.
\end{align*}
This completes the proof.
\end{proof}

\begin{remark}
The replacement of $\|A-A^{-1}\|_2$ in
Lemma~\ref{lemma:FDE_tail_bound} by $\|\log(A)\|_2$ in
Theorem~\ref{thm:DE_logA_error_simple} follows from the different
error decay rates, not a direct comparison of these norms.
For fixed $A\ne I$ and $q\in(0,1)$, put
$L=\|\log(A)\|_2$ and $\alpha=2\pi q d_*(A)$. The proof gives
\[
E_{\mathrm{disc}}=\OO(\E^{-\alpha n}),
\qquad
E_{\mathrm{tail}}=\OO(\E^{-\pi^2n})
\qquad(n\to\infty),
\]
with $0<\alpha<\pi^2$. This faster tail decay absorbs the factor
$\E^L$ in $\|A-A^{-1}\|_2\leq2L\E^L$ when $\pi^2n\geq L$;
for $\pi^2n<L$, the constant sign of the scalar integrands instead
bounds the overall error by $\max\{L,E_{\mathrm{disc}}\}$.
Together, these arguments give the bound for every positive integer $n$.
For $A\ne I$, factoring out $L$ yields a relative-error bound, whose
matrix dependence remains through $d_*(A)$ and $\gamma_{q d_*(A)}$.
\end{remark}

In particular, as described in Section~\ref{subsec:total_error}, when $\lambda_{\min}\lambda_{\max} = 1$ holds, the following theorem is obtained.

\begin{theorem}\label{thm:DE_logA_error}
Let $A$ be a Hermitian positive definite matrix, and assume that its
minimum eigenvalue $\lambda_{\min}$ and maximum eigenvalue
$\lambda_{\max}$ satisfy $\lambda_{\min}\lambda_{\max}=1$.
Let $L=\log(\sqrt{\kappa(A)})$, and let $d_0(\lambda)$ be defined
by~\eqref{eq:d0-simple}.
For each $y\in(0,d_0(\sqrt{\kappa(A)}))$,
let $x_{L,y}$, $T_{L,y}$, $\tau_{L,y}$, and $a_{L,y}$ be defined
by~\eqref{eq:x-L-y}--\eqref{eq:a-L-y}, respectively.
Fix $q\in(0,1)$, independently of $A$ and $h$.
Let $n$ be a positive integer, and let $h$ and $N$ be defined
by~\eqref{eq:h_def} and~\eqref{eq:N_def}, respectively.
Then, we have
\[
\left\|\log(A)-h\sum_{k=-N}^{N}F_{\mathrm{DE}}(kh)\right\|_2
\leq
2L\left\{
\frac{1}{\gamma_{q d_0(\sqrt{\kappa(A)})}}+\E^{2q\cos(d_0(\sqrt{\kappa(A)}))}
\right\}
\frac{\E^{-2\pi q d_0(\sqrt{\kappa(A)})n}}
{1-\E^{-2\pi q d_0(\sqrt{\kappa(A)})n}},
\]
where $\gamma_y$ is defined by~\eqref{eq:gamma_y}.
\end{theorem}

\begin{proof}
The normalization gives $\lambda_{\max}=\lambda_{\min}^{-1}
=\sqrt{\kappa(A)}$, hence
\[
\|\log(A)\|_2=\log(\sqrt{\kappa(A)})=L,
\qquad
d_*(A)=d_0(\E^L)=d_0(\sqrt{\kappa(A)}).
\]
Substitution into Theorem~\ref{thm:DE_logA_error_simple} proves the claim.
\end{proof}

\subsection{Extension to diagonalizable matrices with positive real eigenvalues}
\label{sec:evaluate-total-diagonalizable}
The overall-error bounds extend to diagonalizable matrices with positive
real eigenvalues by a similarity transformation.  In this setting,
the maximum absolute logarithm of the eigenvalues need not equal
$\|\log(A)\|_2$, and the eigenvalue ratio need not equal the spectral
condition number $\kappa(A)$.

\begin{theorem}
\label{thm:DE_logA_error_simple_diagonalizable}
Let $A\in\mathbb{C}^{m\times m}$ be diagonalizable with positive real
eigenvalues, and fix any diagonalization
\[
A=VDV^{-1},\qquad
D=\diag(\lambda_1,\ldots,\lambda_m),
\qquad \lambda_i>0.
\]
Put
\[
L_\Lambda=\max_{\lambda\in\Lambda(A)}|\log\lambda|,
\qquad
d_*(A)=\min_{\lambda\in\Lambda(A)}d_0(\lambda),
\]
where $d_0(\lambda)$ is defined by~\eqref{eq:d0-simple}.
For each $y\in(0,d_*(A))$, define $x_{L_\Lambda,y}$,
$T_{L_\Lambda,y}$, $\tau_{L_\Lambda,y}$, and $a_{L_\Lambda,y}$
by~\eqref{eq:x-L-y}--\eqref{eq:a-L-y} with $L=L_\Lambda$,
and define $\gamma_y$ by~\eqref{eq:gamma_y} with the same choice of $L$.
Fix $q\in(0,1)$, independently of $A$ and $h$.
Let $n$ be a positive integer, and let $h$ and $N$ be defined
by~\eqref{eq:h_def} and~\eqref{eq:N_def}, respectively.
Then,
\[
\left\|\log(A)-h\sum_{k=-N}^{N}F_{\mathrm{DE}}(kh)\right\|_2
\leq
2\kappa(V)\|\log(A)\|_2\left\{
\frac{1}{\gamma_{q d_*(A)}}+\E^{2q\cos(d_*(A))}
\right\}
\frac{\E^{-2\pi q d_*(A)n}}{1-\E^{-2\pi q d_*(A)n}}.
\]
The sharper estimate obtained by replacing the prefactor
$\|\log(A)\|_2$ by $L_\Lambda$ also holds; the definitions of
$d_*(A)$ and $\gamma_y$ remain unchanged.
\end{theorem}

\begin{proof}
Write $F_{\mathrm{DE}}(x;B)$ for~\eqref{eq:FDE-x} with $A$ replaced by
$B$, and put $S_n(B)=h\sum_{k=-N}^{N}F_{\mathrm{DE}}(kh;B)$.
The similarity identities $F_{\mathrm{DE}}(x;A)
=VF_{\mathrm{DE}}(x;D)V^{-1}$ and $\log(A)=V\log(D)V^{-1}$ give
\[
\|\log(A)-S_n(A)\|_2
=\|V\{\log(D)-S_n(D)\}V^{-1}\|_2
\leq\kappa(V)\|\log(D)-S_n(D)\|_2.
\]
Since $D$ is Hermitian positive definite,
$\|\log(D)\|_2=L_\Lambda$ and $d_*(D)=d_*(A)$.
Applying Theorem~\ref{thm:DE_logA_error_simple} to $D$ proves the sharper
estimate; $L_\Lambda\leq\|\log(A)\|_2$, since $\log\lambda_i$ are the
eigenvalues of $\log(A)$, then gives the stated bound.
If $L_\Lambda=0$, diagonalizability implies $A=I$, and both bounds
hold with zero error.
\end{proof}

For the normalized case, we use the eigenvalue ratio
$\kappa_\Lambda(A)=\lambda_{\max}/\lambda_{\min}$ rather than
the spectral condition number $\kappa(A)$.

\begin{theorem}
\label{thm:DE_logA_error_diagonalizable}
Let $A\in\mathbb{C}^{m\times m}$ be diagonalizable with positive real
eigenvalues, and fix a diagonalization $A=VDV^{-1}$.
Assume that the minimum eigenvalue $\lambda_{\min}$ and maximum
eigenvalue $\lambda_{\max}$ of $A$ satisfy
$\lambda_{\min}\lambda_{\max}=1$.
Put $\kappa_\Lambda(A)=\lambda_{\max}/\lambda_{\min}$ and
$L_\Lambda=\log(\sqrt{\kappa_\Lambda(A)})$.
For each $y\in(0,d_0(\sqrt{\kappa_\Lambda(A)}))$, define
$x_{L_\Lambda,y}$, $T_{L_\Lambda,y}$, $\tau_{L_\Lambda,y}$,
and $a_{L_\Lambda,y}$ by~\eqref{eq:x-L-y}--\eqref{eq:a-L-y}
with $L=L_\Lambda$, and define $\gamma_y$ by~\eqref{eq:gamma_y}
with the same choice of $L$.
Fix $q\in(0,1)$, independently of $A$ and $h$.
Let $n$ be a positive integer, and let $h$ and $N$ be defined
by~\eqref{eq:h_def} and~\eqref{eq:N_def}, respectively.
Then,
\[
\left\|\log(A)-h\sum_{k=-N}^{N}F_{\mathrm{DE}}(kh)\right\|_2
\leq
2\kappa(V)\|\log(A)\|_2\left\{
\frac{1}{\gamma_{q d_0(\sqrt{\kappa_\Lambda(A)})}}+\E^{2q\cos(d_0(\sqrt{\kappa_\Lambda(A)}))}
\right\}
\frac{\E^{-2\pi q d_0(\sqrt{\kappa_\Lambda(A)})n}}
{1-\E^{-2\pi q d_0(\sqrt{\kappa_\Lambda(A)})n}}.
\]
The sharper estimate obtained by replacing the prefactor
$\|\log(A)\|_2$ by $L_\Lambda$ also holds.
\end{theorem}

\begin{proof}
The normalization gives $\lambda_{\max}=\lambda_{\min}^{-1}
=\sqrt{\kappa_\Lambda(A)}$, hence
\[
\max_{\lambda\in\Lambda(A)}|\log\lambda|=L_\Lambda,
\qquad
d_*(A)=d_0(\E^{L_\Lambda})=d_0(\sqrt{\kappa_\Lambda(A)}).
\]
Substitution into Theorem~\ref{thm:DE_logA_error_simple_diagonalizable}
proves both estimates.
\end{proof}

\begin{remark}
Both theorems hold for every choice of the diagonalization $V$.
The factor $\kappa(V)$ depends on this choice, including the scaling
of the eigenvectors; it is not the condition number of $A$.
For Hermitian positive definite $A$, one can choose $V$ unitary,
so that $\kappa(V)=1$, $L_\Lambda=\|\log(A)\|_2$, and
$\kappa_\Lambda(A)=\kappa(A)$.  Thus, the original theorems are recovered.
For $A\ne I$, division of the displayed bounds by $\|\log(A)\|_2$
gives relative-error bounds, but their numerical evaluation requires
spectral information and an upper bound for $\kappa(V)$.
The diagonalization is used only in the analysis; evaluating the DE
formula does not require computing $V$ or $D$.

Computing a full eigendecomposition solely to evaluate these bounds
may defeat the computational motivation for quadrature; their primary
purpose here is to establish convergence beyond the Hermitian case.
\end{remark}

\subsection{Practical algorithms: norm-free adaptive and error-bound-based methods}
\label{sec:adaptive-new}
We now describe a norm-free adaptive implementation and a direct
implementation with error-bound-based parameter selection.
For a positive integer $n$, define
\begin{equation}
S_n=\frac{1}{n}\sum_{k=-N}^{N}F_{\mathrm{DE}}\left(\frac{k}{n}\right),
\qquad
N=\left\lceil n\arsinh(\pi n)\right\rceil.
\label{eq:new-adaptive-quadrature-sum}
\end{equation}
In the adaptive implementation, unlike the method in
Section~\ref{sec:adaptive-Tatsuoka}, the mesh is refined and the
truncation interval is enlarged simultaneously.
The parameter $n$ is the reciprocal of the mesh size; the number of
quadrature nodes is $2N+1$.

\paragraph{Evaluation of the integrand}
To avoid exponentially large coefficients in the shifted matrices,
we rewrite the expression in~\eqref{eq:FDE-x} according to the sign of
$x$.  For $x\geq0$, the factor $\E^{-\pi\sinh x}$ is at most one,
whereas for $x<0$ it grows without bound as $x\to-\infty$.
In the latter case, factoring $\E^{-\pi\sinh x}$ out of the shifted
matrix gives an equivalent expression involving $\E^{\pi\sinh x}$,
which is at most one.  Thus, we use
\[
F_{\mathrm{DE}}(x)=
\begin{cases}
(A-I)\left(A+\E^{-\pi\sinh x}I\right)^{-1}
\displaystyle\frac{\pi\cosh x}{1+\E^{\pi\sinh x}}
& (x\geq0),\\[2ex]
(A-I)\left(I+\E^{\pi\sinh x}A\right)^{-1}
\displaystyle\frac{\pi\cosh x}{1+\E^{-\pi\sinh x}}
& (x<0).
\end{cases}
\]
Since $A-I$ commutes with both shifted matrices and their inverses,
these expressions can be evaluated by solving linear systems rather
than explicitly forming inverses.  Specifically, put
$\psi(x)=\E^{-\pi|\sinh x|}$ and, at each quadrature node, solve
\[
\begin{cases}
\left(A+\psi(x)I\right)X(x)=A-I & (x\geq0),\\[1ex]
\left(I+\psi(x)A\right)X(x)=A-I & (x<0).
\end{cases}
\]
Then evaluate
\[
F_{\mathrm{DE}}(x)=\frac{\pi(\cosh x)\psi(x)}{1+\psi(x)}X(x).
\]
The scalar weight is also expressed in terms of $\psi(x)$, so no positive
arguments need to be evaluated in the exponential function.

\paragraph{Reuse of previous evaluations}
When $n$ is doubled, put
\[
N'=\left\lceil 2n\arsinh(2\pi n)\right\rceil.
\]
Since $N'\geq2N$ for $n\geq1$, all previously evaluated nodes remain
in the new quadrature.  Splitting the central sum into even and odd
indices gives the update formula
\begin{align}
S_{2n}
&=\frac12 S_n
+\frac{1}{2n}\sum_{j=-N+1}^{N}
F_{\mathrm{DE}}\left(\frac{2j-1}{2n}\right)
+\frac{1}{2n}\left\{
\sum_{k=-N'}^{-2N-1}F_{\mathrm{DE}}\left(\frac{k}{2n}\right)
+\sum_{k=2N+1}^{N'}F_{\mathrm{DE}}\left(\frac{k}{2n}\right)
\right\}.
\label{eq:new-adaptive-update}
\end{align}
Thus, only the odd-indexed nodes in the old interval and the nodes in
the two newly added tails need to be evaluated.  If each new node is
evaluated once, the total number of integrand evaluations up to a
fixed final value of $n$ is $2N+1$, the same as for a direct computation
at that value.  The adaptive implementation does, however, require
additional summation and stopping tests, and doubling may yield a
larger final $n$ than a directly selected value.

\paragraph{A norm-free stopping criterion}
Let $\widetilde S_n$ denote the approximation computed in
floating-point arithmetic.  For practical computations without
spectral information or matrix norm evaluations, we use successive
approximations as a convergence indicator.  Given an absolute tolerance
$\varepsilon_{\mathrm{abs}}>0$ and a relative tolerance
$\varepsilon_{\mathrm{rel}}>0$, the indicator is satisfied when
\begin{equation}
\left|(\widetilde S_{2n}-\widetilde S_n)_{ij}\right|
\leq\varepsilon_{\mathrm{abs}}
+\varepsilon_{\mathrm{rel}}\left|(\widetilde S_{2n})_{ij}\right|
\qquad\text{for every }i,j.
\label{eq:new-adaptive-stopping}
\end{equation}
The absolute tolerance accommodates entries whose values are close to
zero.  This test uses only entrywise comparisons and does not require
$\|\log(A)\|_2$ or any matrix norm.  Its tolerances control the change
between successive approximations, not a rigorously bounded error in
the spectral norm.  In particular, a small change can result from
rounding-error stagnation, so this criterion alone does not guarantee
the requested accuracy.  We do not apply a fixed Richardson factor to
the difference, since both the mesh size and the truncation interval
change during refinement.

\paragraph{Parameter selection based on the error bound}
For Hermitian positive definite $A\ne I$, spectral information allows
the quadrature parameter to be selected before evaluating the integrand.  With a fixed $q\in(0,1)$, put
\begin{equation*}
\alpha=2\pi q d_*(A),
\qquad
C=2\left\{\frac{1}{\gamma_{q d_*(A)}}+\E^{2q\cos(d_*(A))}\right\},
\qquad
E_n=C\frac{\E^{-\alpha n}}{1-\E^{-\alpha n}}.
\end{equation*}
Theorem~\ref{thm:DE_logA_error_simple} gives
\begin{equation*}
\frac{\|\log(A)-S_n\|_2}{\|\log(A)\|_2}\leq E_n.
\end{equation*}
Hence, $E_n\leq\varepsilon_{\mathrm{rel}}$ is a sufficient condition
for achieving the relative quadrature-error target in exact arithmetic.
Equivalently, it suffices to choose
\begin{equation*}
n\geq n_{\mathrm{req}}:=\max\left\{1,
\left\lceil\frac{1}{\alpha}
\log\left(1+\frac{C}{\varepsilon_{\mathrm{rel}}}\right)
\right\rceil\right\}.
\end{equation*}
Let $M_{\max}$ denote the maximum number of quadrature nodes allowed (the
node limit). Since $S_n$ uses $2\lceil n\arsinh(\pi n)\rceil+1$ nodes,
the node limit imposes $\lceil n\arsinh(\pi n)\rceil\leq K$, where
$K=\lfloor(M_{\max}-1)/2\rfloor$. Let $n_{\max}$ denote the largest
integer satisfying this condition. Rather than numerically solving
this inequality for its largest admissible integer,
Algorithm~\ref{alg:new-adaptive-error-bound} uses the explicit safe
estimate
\[
n'_{\max}=\max\left\{1,
\left\lfloor\frac{K}{\arsinh(\pi K)}\right\rfloor\right\}.
\]
For $K\geq2$, this estimate satisfies $n'_{\max}\leq n_{\max}$:
when the floor is positive, $n'_{\max}\leq K$ and hence
\[
n'_{\max}\arsinh(\pi n'_{\max})
\leq n'_{\max}\arsinh(\pi K)\leq K;
\]
when the floor is zero, $n'_{\max}=1$ and
$\arsinh(\pi)<2\leq K$.  Thus the selected quadrature always respects
the node limit ($M_{\max}$), although this estimate need not use all available
nodes.  The algorithm sets
$n_{\mathrm{stop}}=\min\{n_{\mathrm{req}},n'_{\max}\}$ and computes
$S_{n_{\mathrm{stop}}}$ directly using
\eqref{eq:new-adaptive-quadrature-sum}.  If
$n_{\mathrm{req}}>n'_{\max}$, it returns the approximation at
$n'_{\max}$ with a conservative-limit warning; the error-bound target
is then not certified by this parameter choice, even though a larger
admissible $n$ may exist.
This direct implementation requires no successive-approximation tests
or intermediate quadrature sums and avoids rounding $n_{\mathrm{req}}$
up to a member of a doubling sequence.  Consequently, it may use fewer
quadrature nodes than successive doubling to meet the same error bound.
All quadrature nodes are determined before integrand evaluation, so
their evaluations can be scheduled independently.
The constants require spectral information through
$L=\|\log(A)\|_2$, $d_*(A)$, and $\gamma_{q d_*(A)}$;
these can be evaluated from the extreme eigenvalues of $A$ without
forming $\log(A)$, as noted in
Remark~\ref{rem:L-from-extreme-eigenvalues}.
Thus, this parameter-selection strategy is an optional alternative,
not part of the norm-free implementation.  Theorem~\ref{thm:DE_logA_error} provides
the corresponding bound in terms of $\kappa(A)$ when
$\lambda_{\min}\lambda_{\max}=1$.

Table~\ref{tab:adaptive-information} summarizes the additional matrix
information required by the three implementations, apart from the
matrix $A$ used in integrand evaluations. The error-bound algorithm
listed here uses the Hermitian positive definite bound; none of the three
implementations requires forming $\log(A)$ beforehand.

\begin{table}[tbp]
\centering
\caption{Matrix information required for preprocessing, parameter selection, and stopping tests.}
\label{tab:adaptive-information}
\begin{tabular}{@{}p{0.24\textwidth}p{0.35\textwidth}p{0.34\textwidth}@{}}
\toprule
\raggedright Method & \raggedright Matrix information for preprocessing & \raggedright Parameter selection or stopping test \tabularnewline
\midrule
\raggedright Tatsuoka et al.\ (Algorithm~\ref{alg:adaptive-Tatsuoka})
& \raggedright $\|A-I\|_2$, $\|A^{-1}\|_2$, and $\rho(A)$
& \raggedright $\|\widehat S_{2n}-\widehat S_n\|_2$ \tabularnewline[1ex]
\raggedright Proposed norm-free method (Algorithm~\ref{alg:new-adaptive})
& \raggedright No matrix norms or spectral information
& \raggedright Entrywise comparisons of successive approximations \tabularnewline[1ex]
\raggedright Proposed error-bound method (Algorithm~\ref{alg:new-adaptive-error-bound})
& \raggedright $\lambda_{\min}$ and $\lambda_{\max}$ (Hermitian positive definite $A$)
& \raggedright A priori selection of $n$ from $E_n$; no refinement tests \tabularnewline
\bottomrule
\end{tabular}
\end{table}

\paragraph{Computational procedures and finite-precision limitations}
Algorithms~\ref{alg:new-adaptive} and~\ref{alg:new-adaptive-error-bound}
summarize the norm-free adaptive and error-bound-based direct
implementations, respectively, and assume $A\ne I$. Users should handle
the case $A=I$ separately by returning the zero matrix, which is the
exact value of $\log(I)$.
For simplicity, the pseudocode denotes the computed approximations by
$S_n$ rather than $\widetilde S_n$; the distinction is retained in the
discussion of the norm-free stopping criterion above and in the
finite-precision error analysis below.

\begin{algorithm}[tbp]
\caption{Norm-free adaptive algorithm based on the DE formula with reuse of previous evaluations}
\label{alg:new-adaptive}
\begin{algorithmic}[1]
\Require Matrix $A\ne I$ with no eigenvalues on $(-\infty,0]$.
\Require Positive integer $M_{\max}\geq 5$;
 tolerances $\varepsilon_{\mathrm{abs}},\varepsilon_{\mathrm{rel}}>0$.
\Ensure Approximation to $\log(A)$ and a termination status.
\State $n\gets 1$
\State $N\gets \lceil n\arsinh(\pi n)\rceil$
\State Compute $S_1$ using~\eqref{eq:new-adaptive-quadrature-sum}.
\State $N'\gets \lceil 2n\arsinh(2\pi n)\rceil$
\While{$2N'+1\leq M_{\max}$}
  \State Compute $S_{2n}$ using~\eqref{eq:new-adaptive-update}.
  \If{\eqref{eq:new-adaptive-stopping} holds}
    \State \Return $(S_{2n},\text{successive-approximation indicator satisfied})$
  \EndIf
  \State $n\gets 2n$
  \State $N \gets N'$
  \State $N'\gets \lceil 2n\arsinh(2\pi n)\rceil$
\EndWhile
\State \Return $(\text{last computed approximation},\text{refinement-limit warning})$
\end{algorithmic}
\end{algorithm}

\begin{algorithm}[tbp]
\caption{Algorithm based on the DE formula with error-bound-based parameter selection and direct computation}
\label{alg:new-adaptive-error-bound}
\begin{algorithmic}[1]
\Require Hermitian positive definite matrix $A\ne I$.
\Require Positive integer $M_{\max}\geq 5$;
 relative tolerance $\varepsilon_{\mathrm{rel}}>0$.
\Ensure Approximation to $\log(A)$ and a termination status.
\State Compute the extreme eigenvalues $\lambda_{\min}$ and $\lambda_{\max}$ of $A$.
\State $L\gets \max\{\log(1/\lambda_{\min}), \log(\lambda_{\max})\}$
\State $d_*(A)\gets \arccos(\sqrt{2/(1+\sqrt{1+(2\pi/L)^2})})$
\State $q\gets 0.95$, $y\gets q d_*(A)$
\State $x_{L,y}\gets\arsinh(L/(\pi\cos y))$
\State $T_{L,y}\gets(\pi\cos y/\tan y - L)/2$
\State $\tau_{L,y}\gets T_{L,y}/(1 + \sqrt{T_{L,y}})$
\State $a_{L,y}\gets (\pi/2)\sin y\sqrt{1+\{\sinh x_{L,y}+2\tau_{L,y}/(\pi\cos y)\}^2}$
\State $\gamma_y\gets\cos\left((\pi/2)\sin y\right)\cos(a_{L,y})\cos y$
\State $\alpha\gets2\pi y$, $C\gets2\{(1/\gamma_y)+\exp(2q\cos(d_*(A)))\}$
\State $n_{\mathrm{req}}\gets\max\{1,\lceil\log(1+C/\varepsilon_{\mathrm{rel}})/\alpha\rceil\}$
\State $K\gets \lfloor(M_{\max}-1)/2\rfloor$, $n'_{\max} \gets \max\{1,\lfloor K/\arsinh(\pi K)\rfloor\}$
\State $n_{\mathrm{stop}}\gets\min\{n_{\mathrm{req}},n'_{\max}\}$
\If{$n_{\mathrm{req}}\leq n'_{\max}$}
  \State $\mathrm{status}\gets\text{quadrature-error target satisfied}$
\Else
  \State $\mathrm{status}\gets\text{conservative-limit warning}$
\EndIf
\State Compute $S_{n_{\mathrm{stop}}}$ directly using~\eqref{eq:new-adaptive-quadrature-sum}.
\State \Return $(S_{n_{\mathrm{stop}}},\mathrm{status})$
\end{algorithmic}
\end{algorithm}
For the error-bound implementation, the guarantee concerns $S_n$,
not automatically $\widetilde S_n$.  Indeed,
\[
\frac{\|\log(A)-\widetilde S_n\|_2}{\|\log(A)\|_2}
\leq E_n+
\frac{\|S_n-\widetilde S_n\|_2}{\|\log(A)\|_2}.
\]
The second term includes errors in integrand evaluation and summation,
and is not bounded by the present analysis.  A fully certified
floating-point result would additionally require a rigorous bound for
this term and validated evaluation of the constants in $E_n$.
Increasing $n$ controls the quadrature bound but cannot alone certify
or eliminate these finite-precision errors.

\subsection{Proof of Lemma~\ref{lem:bound-None}}
\label{sec:proof-essential-lemma}

Analyticity and finiteness of
$\mathcal{N}(f_{\mathrm{DE}}(\cdot,\lambda),d)$ for
$0<d<d_0(\lambda)$ follow from Proposition~\ref{prop:analytic_strip}
and the derivation of~\eqref{eq:4-23}. We now obtain an explicit,
computable bound for this quantity. By~\eqref{eq:scalar-fDE},
\begin{align*}
|f_{\mathrm{DE}}(x+\I y,\lambda)|
= \frac{|\lambda - 1|\pi|\cosh (x+\I y)|}{|\lambda + \E^{-\pi\sinh(x+\I y)}||1 + \E^{\pi\sinh(x+\I y)}|}
\leq \frac{|\lambda - 1|\pi\cosh x}{|\lambda + \E^{-\pi\sinh(x+\I y)}||1 + \E^{\pi\sinh(x+\I y)}|}.
\end{align*}
The remaining terms are
$|\lambda + \E^{-\pi\sinh(x+\I y)}|$
and $|1 + \E^{\pi\sinh(x+\I y)}|$
in the denominator.
For the second denominator,
the following lemma provides the desired estimate.

\begin{lemma}[Okayama et al.~{\cite[Lemma 4.22]{Okayama-et-al}}]
\label{Lem:Bound-DEt-essential}
Let $x$ and $y$ be real numbers with $|y|<\pi/2$.
Then,
\begin{align}
\label{phi-div-denominator-plus}
     \left|
      \frac{1}{1+\E^{\pi\sinh(x + \I y)}}
     \right|
&\leq\frac{1}{(1+\E^{\pi\sinh(x)\cos y})\cos((\pi/2)\sin y)}.
\end{align}
\end{lemma}

For the first denominator, note that
\begin{align*}
\frac{1}{|\lambda + \E^{-\pi\sinh(x+\I y)}|}
=\frac{1}{|\lambda||1 + \E^{-\log(\lambda)-\pi\sinh(x+\I y)}|}
&=\frac{1}{|\lambda||1 + \E^{-|\log(\lambda)|-\pi\sinh(x+\I y)}|}
\intertext{holds for $\lambda\geq 1$, and}
\frac{1}{|\lambda + \E^{-\pi\sinh(x+\I y)}|}
=\frac{|\E^{\pi\sinh(x+\I y)}|}{|\lambda \E^{\pi\sinh(x+\I y)} + 1|}
\leq \frac{\E^{\pi\sinh(x)\cos y}}{|\E^{\log(\lambda)+\pi\sinh(x+\I y)} + 1|}
&=\frac{\E^{\pi\sinh(x)\cos y}}{|1 + \E^{-|\log(\lambda)|-\pi\sinh(-(x+\I y))}|}
\end{align*}
holds for $0<\lambda<1$.
Therefore, for both cases,
we should estimate
$|1 + \E^{-|\log(\lambda)|-\pi\sinh\zeta}|$
for $\zeta\in\domD_d$ with some $d>0$.
For this purpose,
the following result is useful.
\begin{lemma}[Okayama~{\cite[Lemma 4.12]{OkayamaDEunilateral}}]
\label{lem:important-estimate-general-L}
Let $L$ be a positive constant, and let $d_L$ be a positive
constant defined by
\begin{equation}
d_L = \arccos\left(\sqrt{\frac{2}{1+\sqrt{1+(2\pi/L)^2}}}\right).
\label{eq:d_L}
\end{equation}
For each $y\in(0,d_L)$,
let $x_{L,y}$, $T_{L,y}$, $\tau_{L,y}$, and
$a_{L,y}$ be defined
by~\eqref{eq:x-L-y}--\eqref{eq:a-L-y}, respectively.
Then, for all $x\in\mathbb{R}$ and $y\in(-d_L,d_L)$,
\begin{equation*}
 \frac{1}{|1+\E^{-L-\pi\sinh(x+\I y)}|}
 \leq
 \frac{1}
 {(1+\E^{-L-\pi\sinh(x)\cos y})c_{L,y}},
\end{equation*}
where
\[
 c_{L,y}=
 \begin{cases}
 1
 & (y = 0),\\[2ex]
 \cos a_{L,|y|}
 & (0<|y|<d_L).
 \end{cases}
\]
\end{lemma}

We apply this lemma with $L=|\log\lambda|$ to prove
Lemma~\ref{lem:bound-None}.
\begin{proof}[Lemma~\ref{lem:bound-None}]
If $\lambda=1$, then $f_{\mathrm{DE}}(\cdot,1)=0$, and the result is
immediate.  We therefore assume that $\lambda\ne1$, so that $L>0$.
Equations~\eqref{eq:d0-simple} and~\eqref{eq:d_L} give
$d_0(\lambda)=d_L$, so $0<d<d_L$.  Proposition~\ref{prop:analytic_strip} shows that
$f_{\mathrm{DE}}(\cdot,\lambda)$ is analytic in $\domD_d$.
Moreover, the double-exponential decay used in the derivation
of~\eqref{eq:4-23} is uniform on every closed substrip of
$\domD_{d_0(\lambda)}$.  Hence,
\[
 \lim_{x\to\pm\infty}
 \int_{-d}^{d}|f_{\mathrm{DE}}(x+\I y,\lambda)|\D y=0.
\]
It remains to estimate
$\mathcal{N}(f_{\mathrm{DE}}(\cdot,\lambda),d)$.

Fix $\eta\in(0,d)$ and put
\[
 s=\pi\sinh x\cos\eta,
 \qquad
 \gamma_\eta
 =\cos\left(\frac{\pi}{2}\sin\eta\right)
 \cos(a_{L,\eta})\cos\eta.
\]
All the factors in $\gamma_\eta$ are positive because
$0<\eta<d<d_L\leq\pi/2$.
First, suppose that $\lambda>1$.  Then, $L=\log\lambda$, and
Lemmas~\ref{Lem:Bound-DEt-essential}
and~\ref{lem:important-estimate-general-L} give,
for both $x+\I\eta$ and $x-\I\eta$,
\begin{align*}
 |f_{\mathrm{DE}}(x\mathbin{\pm}\I\eta,\lambda)|
 &\leq
 \frac{|\lambda-1|\pi\cosh x}
 {\lambda
 \cos((\pi/2)\sin\eta)\cos(a_{L,\eta})}
 \cdot
 \frac{1}
 {(1+\E^{-L-s})(1+\E^s)}.
\end{align*}
Using the change of variables $u=\E^s$, we obtain
\begin{align*}
 \int_{-\infty}^{\infty}
 |f_{\mathrm{DE}}(x\mathbin{\pm}\I\eta,\lambda)|\D x
 &\leq
 \frac{|\lambda-1|}{\lambda\gamma_\eta}
 \int_0^\infty
 \frac{\D u}{(u+\E^{-L})(u+1)}
 =\frac{|\lambda-1|}{\lambda\gamma_\eta}\cdot
 \frac{L}{1-\E^{-L}}
 =\frac{L}{\gamma_\eta}.
\end{align*}

Next, suppose that $0<\lambda<1$.  Then, $L=-\log\lambda$ and
$\E^{-L}=\lambda$.  Using the estimate derived above for the first
denominator with $\zeta=-(x\mathbin{\pm}\I\eta)$, together with
Lemma~\ref{Lem:Bound-DEt-essential}, we have
\begin{align*}
 |f_{\mathrm{DE}}(x\mathbin{\pm}\I\eta,\lambda)|
 &\leq
 \frac{|\lambda-1|\pi\cosh x}
 {\cos((\pi/2)\sin\eta)\cos(a_{L,\eta})}\cdot
 \frac{\E^s}{(1+\lambda\E^s)(1+\E^s)}.
\end{align*}
The same change of variables gives
\begin{align*}
 \int_{-\infty}^{\infty}
 |f_{\mathrm{DE}}(x\mathbin{\pm}\I\eta,\lambda)|\D x
 &\leq
 \frac{1-\lambda}{\gamma_\eta}
 \int_0^\infty
 \frac{\D u}{(1+\lambda u)(1+u)}
 =\frac{1-\lambda}{\gamma_\eta}\cdot
 \frac{L}{1-\lambda}
 =\frac{L}{\gamma_\eta}.
\end{align*}
Thus, in both cases,
\[
 \int_{-\infty}^{\infty}
 \left\{
 |f_{\mathrm{DE}}(x+\I\eta,\lambda)|
 +|f_{\mathrm{DE}}(x-\I\eta,\lambda)|
 \right\}\D x
 \leq\frac{2L}{\gamma_\eta}.
\]
Letting $\eta\to d-0$ and using the continuity and positivity of
$\gamma_\eta$ on $(0,d_L)$, we conclude that
\begin{align*}
 \mathcal{N}(f_{\mathrm{DE}}(\cdot,\lambda),d)
 &\leq\frac{2L}{\gamma_d}.
\end{align*}
Therefore,
$f_{\mathrm{DE}}(\cdot,\lambda)\in\mathbf{B}(\domD_d)$, and the
claimed estimate follows.
\end{proof}

\section{Numerical examples}
\label{sec:numerical}

In this section, we present numerical experiments comparing
Algorithm~\ref{alg:adaptive-Tatsuoka}, the norm-free
Algorithm~\ref{alg:new-adaptive}, and the error-bound-based
Algorithm~\ref{alg:new-adaptive-error-bound}. We assess their accuracy
as a function of both the number of quadrature nodes and computation
time, compare their total termination times, and evaluate the error
bound reported by Algorithm~\ref{alg:new-adaptive-error-bound}.

\subsection{Experimental setup}
\label{subsec:numerical-setup}

The source code and numerical data used in these experiments are available
at the GitHub repository \url{https://github.com/okayamat/matrix-log-DE}.
The experiments were performed on a MacBook Air with an Apple M4 chip
and 24 GB of memory, running macOS Tahoe 26.6.2 and MATLAB R2024b.
MATLAB's \texttt{maxNumCompThreads} function reported a maximum of 10
computational threads; this value was queried and not explicitly set.
Execution times were measured using MATLAB's \texttt{timeit} function.
Each timed call invokes the complete algorithm routine, so the measured
time includes all preprocessing performed by the routine, including
any norm or spectral-information calculations, parameter selection,
quadrature evaluations, and any stopping tests. Construction of the test
matrices and reference logarithms is performed outside the timed calls.

Following Tatsuoka et al.~\cite{TatsuEn}, we generated $A$ as follows.
For each matrix dimension $m$, we generated an $m\times m$ matrix $G$ with
independent standard normally distributed entries using MATLAB's
\texttt{randn(m)}.  We computed its QR
factorization $G=QR$ and replaced $Q$ by $Q\diag(s)$,
where $s_k=\operatorname{sign}(R_{kk})$ and $s_k=1$ if $R_{kk}=0$.
This sign correction gives a Haar-uniform random orthogonal matrix.
For each matrix dimension, we generated a single matrix $Q$ and reused
it for all three prescribed condition numbers,
$\kappa(A)=10^1,\,10^4,\,10^7$.
For each prescribed condition number $\kappa(A)$, we set
$r=\kappa(A)^{1/(m-1)}$ and define
$d_j=\kappa(A)^{-1/2}r^{j-1}$ for $j=1,\,\ldots,\,m$; then
$D=\diag(d_1,\,\ldots,\,d_m)$ is positive definite and has
condition number $\kappa(A)$.  The test matrix is
$A=QDQ^{\mathsf{T}}$.
Since $Q$ is orthogonal and $D$ is positive diagonal, $A$ is symmetric
positive definite (hence Hermitian positive definite), with condition
number $\kappa(A)$.
By construction, $\lambda_{\min}\lambda_{\max}=1$ holds.

We report the relative error
$\|\log(A)-S\|_2/\|\log(A)\|_2$ for each computed approximation $S$.
As the reference value for $\log(A)$, we computed
$Q\diag(\log(d_1),\,\ldots,\,\log(d_m))Q^{\mathsf{T}}$.
The matrix dimensions used in the main numerical experiments were
$m=50,100,200,400,800,1600$.  The separate error-bound experiment was
conducted only for $m=50$, since the theoretical bound has no explicit dependence on $m$.
For the experiments,
the parameters for Algorithm~\ref{alg:adaptive-Tatsuoka} were
$\varepsilon=2^{-53}$, $\tau=10^{-11}$, and $n_0=2$; those for
Algorithm~\ref{alg:new-adaptive} were $\varepsilon_{\mathrm{abs}}=10^{-13}$ and
$\varepsilon_{\mathrm{rel}}=10^{-11}$; and that for
Algorithm~\ref{alg:new-adaptive-error-bound} was $\varepsilon_{\mathrm{rel}}=10^{-11}$.
We set the node limit to $M_{\max}=640$ except in the time-error plots
(Figs.~\ref{fig:TimeError-1e1}--\ref{fig:TimeError-1e7}), where we used $M_{\max}=5\cdot2^k$ for $k=0,\,\ldots,\,7$ to compare error--time trade-offs.

\subsection{Accuracy versus number of quadrature nodes}
\label{subsec:numerical-node-error}

\begin{figure}[htpb]
\centering
\scalebox{0.72}{\begin{tikzpicture}[gnuplot]
\tikzset{every node/.append style={font={\ttfamily\fontsize{12.0pt}{14.4pt}\selectfont}}}
\gpmonochromelines
\path (0.000,0.000) rectangle (12.500,8.750);
\gpcolor{color=gp lt color border}
\gpsetlinetype{gp lt border}
\gpsetdashtype{gp dt solid}
\gpsetlinewidth{1.00}
\draw[gp path] (2.027,1.184)--(2.117,1.184);
\draw[gp path] (11.836,1.184)--(11.746,1.184);
\draw[gp path] (2.027,1.629)--(2.207,1.629);
\draw[gp path] (11.836,1.629)--(11.656,1.629);
\node[gp node right] at (1.806,1.629) {$10^{-14}$};
\draw[gp path] (2.027,2.074)--(2.117,2.074);
\draw[gp path] (11.836,2.074)--(11.746,2.074);
\draw[gp path] (2.027,2.520)--(2.207,2.520);
\draw[gp path] (11.836,2.520)--(11.656,2.520);
\node[gp node right] at (1.806,2.520) {$10^{-12}$};
\draw[gp path] (2.027,2.965)--(2.117,2.965);
\draw[gp path] (11.836,2.965)--(11.746,2.965);
\draw[gp path] (2.027,3.410)--(2.207,3.410);
\draw[gp path] (11.836,3.410)--(11.656,3.410);
\node[gp node right] at (1.806,3.410) {$10^{-10}$};
\draw[gp path] (2.027,3.855)--(2.117,3.855);
\draw[gp path] (11.836,3.855)--(11.746,3.855);
\draw[gp path] (2.027,4.300)--(2.207,4.300);
\draw[gp path] (11.836,4.300)--(11.656,4.300);
\node[gp node right] at (1.806,4.300) {$10^{-8}$};
\draw[gp path] (2.027,4.746)--(2.117,4.746);
\draw[gp path] (11.836,4.746)--(11.746,4.746);
\draw[gp path] (2.027,5.191)--(2.207,5.191);
\draw[gp path] (11.836,5.191)--(11.656,5.191);
\node[gp node right] at (1.806,5.191) {$10^{-6}$};
\draw[gp path] (2.027,5.636)--(2.117,5.636);
\draw[gp path] (11.836,5.636)--(11.746,5.636);
\draw[gp path] (2.027,6.081)--(2.207,6.081);
\draw[gp path] (11.836,6.081)--(11.656,6.081);
\node[gp node right] at (1.806,6.081) {$10^{-4}$};
\draw[gp path] (2.027,6.526)--(2.117,6.526);
\draw[gp path] (11.836,6.526)--(11.746,6.526);
\draw[gp path] (2.027,6.971)--(2.207,6.971);
\draw[gp path] (11.836,6.971)--(11.656,6.971);
\node[gp node right] at (1.806,6.971) {$10^{-2}$};
\draw[gp path] (2.027,7.417)--(2.117,7.417);
\draw[gp path] (11.836,7.417)--(11.746,7.417);
\draw[gp path] (2.027,7.862)--(2.207,7.862);
\draw[gp path] (11.836,7.862)--(11.656,7.862);
\node[gp node right] at (1.806,7.862) {$10^{0}$};
\draw[gp path] (2.027,8.307)--(2.117,8.307);
\draw[gp path] (11.836,8.307)--(11.746,8.307);
\draw[gp path] (2.027,1.184)--(2.027,1.364);
\draw[gp path] (2.027,8.307)--(2.027,8.127);
\node[gp node center] at (2.027,0.814) {$0$};
\draw[gp path] (3.428,1.184)--(3.428,1.364);
\draw[gp path] (3.428,8.307)--(3.428,8.127);
\node[gp node center] at (3.428,0.814) {$50$};
\draw[gp path] (4.830,1.184)--(4.830,1.364);
\draw[gp path] (4.830,8.307)--(4.830,8.127);
\node[gp node center] at (4.830,0.814) {$100$};
\draw[gp path] (6.231,1.184)--(6.231,1.364);
\draw[gp path] (6.231,8.307)--(6.231,8.127);
\node[gp node center] at (6.231,0.814) {$150$};
\draw[gp path] (7.632,1.184)--(7.632,1.364);
\draw[gp path] (7.632,8.307)--(7.632,8.127);
\node[gp node center] at (7.632,0.814) {$200$};
\draw[gp path] (9.033,1.184)--(9.033,1.364);
\draw[gp path] (9.033,8.307)--(9.033,8.127);
\node[gp node center] at (9.033,0.814) {$250$};
\draw[gp path] (10.435,1.184)--(10.435,1.364);
\draw[gp path] (10.435,8.307)--(10.435,8.127);
\node[gp node center] at (10.435,0.814) {$300$};
\draw[gp path] (11.836,1.184)--(11.836,1.364);
\draw[gp path] (11.836,8.307)--(11.836,8.127);
\node[gp node center] at (11.836,0.814) {$350$};
\draw[gp path] (2.027,8.307)--(2.027,1.184)--(11.836,1.184)--(11.836,8.307)--cycle;
\node[gp node center,rotate=-270] at (0.295,4.745) {Relative error (2-norm)};
\node[gp node center] at (6.931,0.259) {Number of quadrature nodes ($M$)};
\node[gp node right] at (10.109,7.902) {Algorithm 1 ($m=50$)};
\draw[gp path] (10.330,7.902)--(11.394,7.902);
\draw[gp path] (2.111,7.915)--(2.167,7.524)--(2.279,6.832)--(2.503,5.232)--(2.952,2.356)%
  --(3.849,1.475);
\gpsetpointsize{4.00}
\gppoint{gp mark 8}{(2.111,7.915)}
\gppoint{gp mark 8}{(2.167,7.524)}
\gppoint{gp mark 8}{(2.279,6.832)}
\gppoint{gp mark 8}{(2.503,5.232)}
\gppoint{gp mark 8}{(2.952,2.356)}
\gppoint{gp mark 8}{(3.849,1.475)}
\gppoint{gp mark 8}{(10.862,7.902)}
\node[gp node right] at (10.109,7.452) {Algorithm 2 ($m=50$)};
\draw[gp path] (10.330,7.452)--(11.394,7.452);
\draw[gp path] (2.167,7.065)--(2.391,5.930)--(2.784,3.352)--(3.849,1.478)--(6.203,1.477);
\gppoint{gp mark 6}{(2.167,7.065)}
\gppoint{gp mark 6}{(2.391,5.930)}
\gppoint{gp mark 6}{(2.784,3.352)}
\gppoint{gp mark 6}{(3.849,1.478)}
\gppoint{gp mark 6}{(6.203,1.477)}
\gppoint{gp mark 6}{(10.862,7.452)}
\node[gp node right] at (10.109,7.002) {Algorithm 3 ($m=50$)};
\draw[gp path] (10.330,7.002)--(11.394,7.002);
\gppoint{gp mark 4}{(3.288,1.473)}
\gppoint{gp mark 4}{(10.862,7.002)}
\node[gp node right] at (10.109,6.552) {Algorithm 1 ($m=1600$)};
\draw[gp path] (10.330,6.552)--(11.394,6.552);
\draw[gp path] (2.111,7.915)--(2.167,7.524)--(2.279,6.832)--(2.503,5.232)--(2.952,2.356)%
  --(3.849,1.603);
\gppoint{gp mark 9}{(2.111,7.915)}
\gppoint{gp mark 9}{(2.167,7.524)}
\gppoint{gp mark 9}{(2.279,6.832)}
\gppoint{gp mark 9}{(2.503,5.232)}
\gppoint{gp mark 9}{(2.952,2.356)}
\gppoint{gp mark 9}{(3.849,1.603)}
\gppoint{gp mark 9}{(10.862,6.552)}
\node[gp node right] at (10.109,6.102) {Algorithm 2 ($m=1600$)};
\draw[gp path] (10.330,6.102)--(11.394,6.102);
\draw[gp path] (2.167,7.065)--(2.391,5.930)--(2.784,3.354)--(3.849,1.600)--(6.203,1.601);
\gppoint{gp mark 7}{(2.167,7.065)}
\gppoint{gp mark 7}{(2.391,5.930)}
\gppoint{gp mark 7}{(2.784,3.354)}
\gppoint{gp mark 7}{(3.849,1.600)}
\gppoint{gp mark 7}{(6.203,1.601)}
\gppoint{gp mark 7}{(10.862,6.102)}
\node[gp node right] at (10.109,5.652) {Algorithm 3 ($m=1600$)};
\draw[gp path] (10.330,5.652)--(11.394,5.652);
\gppoint{gp mark 5}{(3.288,1.600)}
\gppoint{gp mark 5}{(10.862,5.652)}
\node[gp node right] at (10.109,5.202) {Error bound for Algorithm 3};
\gpsetdashtype{gp dt 2}
\draw[gp path] (10.330,5.202)--(11.394,5.202);
\draw[gp path] (2.167,7.933)--(2.391,6.796)--(2.559,5.659)--(2.784,4.522)--(3.064,3.385)%
  --(3.288,2.248);
\gpsetdashtype{gp dt solid}
\draw[gp path] (2.027,8.307)--(2.027,1.184)--(11.836,1.184)--(11.836,8.307)--cycle;
\gpdefrectangularnode{gp plot 1}{\pgfpoint{2.027cm}{1.184cm}}{\pgfpoint{11.836cm}{8.307cm}}
\end{tikzpicture}
}
\caption{Relative errors with respect to $M$ for $\log(A)$ with $\kappa(A)=10^1$.}
\label{fig:PointError-1e1}
\scalebox{0.72}{\begin{tikzpicture}[gnuplot]
\tikzset{every node/.append style={font={\ttfamily\fontsize{12.0pt}{14.4pt}\selectfont}}}
\gpmonochromelines
\path (0.000,0.000) rectangle (12.500,8.750);
\gpcolor{color=gp lt color border}
\gpsetlinetype{gp lt border}
\gpsetdashtype{gp dt solid}
\gpsetlinewidth{1.00}
\draw[gp path] (2.027,1.184)--(2.117,1.184);
\draw[gp path] (11.836,1.184)--(11.746,1.184);
\draw[gp path] (2.027,1.629)--(2.207,1.629);
\draw[gp path] (11.836,1.629)--(11.656,1.629);
\node[gp node right] at (1.806,1.629) {$10^{-14}$};
\draw[gp path] (2.027,2.074)--(2.117,2.074);
\draw[gp path] (11.836,2.074)--(11.746,2.074);
\draw[gp path] (2.027,2.520)--(2.207,2.520);
\draw[gp path] (11.836,2.520)--(11.656,2.520);
\node[gp node right] at (1.806,2.520) {$10^{-12}$};
\draw[gp path] (2.027,2.965)--(2.117,2.965);
\draw[gp path] (11.836,2.965)--(11.746,2.965);
\draw[gp path] (2.027,3.410)--(2.207,3.410);
\draw[gp path] (11.836,3.410)--(11.656,3.410);
\node[gp node right] at (1.806,3.410) {$10^{-10}$};
\draw[gp path] (2.027,3.855)--(2.117,3.855);
\draw[gp path] (11.836,3.855)--(11.746,3.855);
\draw[gp path] (2.027,4.300)--(2.207,4.300);
\draw[gp path] (11.836,4.300)--(11.656,4.300);
\node[gp node right] at (1.806,4.300) {$10^{-8}$};
\draw[gp path] (2.027,4.746)--(2.117,4.746);
\draw[gp path] (11.836,4.746)--(11.746,4.746);
\draw[gp path] (2.027,5.191)--(2.207,5.191);
\draw[gp path] (11.836,5.191)--(11.656,5.191);
\node[gp node right] at (1.806,5.191) {$10^{-6}$};
\draw[gp path] (2.027,5.636)--(2.117,5.636);
\draw[gp path] (11.836,5.636)--(11.746,5.636);
\draw[gp path] (2.027,6.081)--(2.207,6.081);
\draw[gp path] (11.836,6.081)--(11.656,6.081);
\node[gp node right] at (1.806,6.081) {$10^{-4}$};
\draw[gp path] (2.027,6.526)--(2.117,6.526);
\draw[gp path] (11.836,6.526)--(11.746,6.526);
\draw[gp path] (2.027,6.971)--(2.207,6.971);
\draw[gp path] (11.836,6.971)--(11.656,6.971);
\node[gp node right] at (1.806,6.971) {$10^{-2}$};
\draw[gp path] (2.027,7.417)--(2.117,7.417);
\draw[gp path] (11.836,7.417)--(11.746,7.417);
\draw[gp path] (2.027,7.862)--(2.207,7.862);
\draw[gp path] (11.836,7.862)--(11.656,7.862);
\node[gp node right] at (1.806,7.862) {$10^{0}$};
\draw[gp path] (2.027,8.307)--(2.117,8.307);
\draw[gp path] (11.836,8.307)--(11.746,8.307);
\draw[gp path] (2.027,1.184)--(2.027,1.364);
\draw[gp path] (2.027,8.307)--(2.027,8.127);
\node[gp node center] at (2.027,0.814) {$0$};
\draw[gp path] (3.428,1.184)--(3.428,1.364);
\draw[gp path] (3.428,8.307)--(3.428,8.127);
\node[gp node center] at (3.428,0.814) {$50$};
\draw[gp path] (4.830,1.184)--(4.830,1.364);
\draw[gp path] (4.830,8.307)--(4.830,8.127);
\node[gp node center] at (4.830,0.814) {$100$};
\draw[gp path] (6.231,1.184)--(6.231,1.364);
\draw[gp path] (6.231,8.307)--(6.231,8.127);
\node[gp node center] at (6.231,0.814) {$150$};
\draw[gp path] (7.632,1.184)--(7.632,1.364);
\draw[gp path] (7.632,8.307)--(7.632,8.127);
\node[gp node center] at (7.632,0.814) {$200$};
\draw[gp path] (9.033,1.184)--(9.033,1.364);
\draw[gp path] (9.033,8.307)--(9.033,8.127);
\node[gp node center] at (9.033,0.814) {$250$};
\draw[gp path] (10.435,1.184)--(10.435,1.364);
\draw[gp path] (10.435,8.307)--(10.435,8.127);
\node[gp node center] at (10.435,0.814) {$300$};
\draw[gp path] (11.836,1.184)--(11.836,1.364);
\draw[gp path] (11.836,8.307)--(11.836,8.127);
\node[gp node center] at (11.836,0.814) {$350$};
\draw[gp path] (2.027,8.307)--(2.027,1.184)--(11.836,1.184)--(11.836,8.307)--cycle;
\node[gp node center,rotate=-270] at (0.295,4.745) {Relative error (2-norm)};
\node[gp node center] at (6.931,0.259) {Number of quadrature nodes ($M$)};
\node[gp node right] at (10.109,7.902) {Algorithm 1 ($m=50$)};
\draw[gp path] (10.330,7.902)--(11.394,7.902);
\draw[gp path] (2.111,7.724)--(2.167,7.652)--(2.279,7.216)--(2.503,6.496)--(2.952,4.799)%
  --(3.849,1.807)--(5.642,1.782);
\gpsetpointsize{4.00}
\gppoint{gp mark 8}{(2.111,7.724)}
\gppoint{gp mark 8}{(2.167,7.652)}
\gppoint{gp mark 8}{(2.279,7.216)}
\gppoint{gp mark 8}{(2.503,6.496)}
\gppoint{gp mark 8}{(2.952,4.799)}
\gppoint{gp mark 8}{(3.849,1.807)}
\gppoint{gp mark 8}{(5.642,1.782)}
\gppoint{gp mark 8}{(10.862,7.902)}
\node[gp node right] at (10.109,7.452) {Algorithm 2 ($m=50$)};
\draw[gp path] (10.330,7.452)--(11.394,7.452);
\draw[gp path] (2.167,7.392)--(2.391,6.665)--(2.784,5.465)--(3.849,2.772)--(6.203,1.777)%
  --(11.584,1.775);
\gppoint{gp mark 6}{(2.167,7.392)}
\gppoint{gp mark 6}{(2.391,6.665)}
\gppoint{gp mark 6}{(2.784,5.465)}
\gppoint{gp mark 6}{(3.849,2.772)}
\gppoint{gp mark 6}{(6.203,1.777)}
\gppoint{gp mark 6}{(11.584,1.775)}
\gppoint{gp mark 6}{(10.862,7.452)}
\node[gp node right] at (10.109,7.002) {Algorithm 3 ($m=50$)};
\draw[gp path] (10.330,7.002)--(11.394,7.002);
\gppoint{gp mark 4}{(4.409,1.776)}
\gppoint{gp mark 4}{(10.862,7.002)}
\node[gp node right] at (10.109,6.552) {Algorithm 1 ($m=1600$)};
\draw[gp path] (10.330,6.552)--(11.394,6.552);
\draw[gp path] (2.111,7.725)--(2.167,7.652)--(2.279,7.216)--(2.503,6.496)--(2.952,4.807)%
  --(3.849,2.091)--(5.642,2.087);
\gppoint{gp mark 9}{(2.111,7.725)}
\gppoint{gp mark 9}{(2.167,7.652)}
\gppoint{gp mark 9}{(2.279,7.216)}
\gppoint{gp mark 9}{(2.503,6.496)}
\gppoint{gp mark 9}{(2.952,4.807)}
\gppoint{gp mark 9}{(3.849,2.091)}
\gppoint{gp mark 9}{(5.642,2.087)}
\gppoint{gp mark 9}{(10.862,6.552)}
\node[gp node right] at (10.109,6.102) {Algorithm 2 ($m=1600$)};
\draw[gp path] (10.330,6.102)--(11.394,6.102);
\draw[gp path] (2.167,7.392)--(2.391,6.665)--(2.784,5.465)--(3.849,2.816)--(6.203,2.087)%
  --(11.584,2.088);
\gppoint{gp mark 7}{(2.167,7.392)}
\gppoint{gp mark 7}{(2.391,6.665)}
\gppoint{gp mark 7}{(2.784,5.465)}
\gppoint{gp mark 7}{(3.849,2.816)}
\gppoint{gp mark 7}{(6.203,2.087)}
\gppoint{gp mark 7}{(11.584,2.088)}
\gppoint{gp mark 7}{(10.862,6.102)}
\node[gp node right] at (10.109,5.652) {Algorithm 3 ($m=1600$)};
\draw[gp path] (10.330,5.652)--(11.394,5.652);
\gppoint{gp mark 5}{(4.409,2.087)}
\gppoint{gp mark 5}{(10.862,5.652)}
\node[gp node right] at (10.109,5.202) {Error bound for Algorithm 3};
\gpsetdashtype{gp dt 2}
\draw[gp path] (10.330,5.202)--(11.394,5.202);
\draw[gp path] (2.167,8.195)--(2.391,7.574)--(2.559,6.960)--(2.784,6.347)--(3.064,5.733)%
  --(3.288,5.120)--(3.568,4.506)--(3.849,3.893)--(4.129,3.279)--(4.409,2.666);
\gpsetdashtype{gp dt solid}
\draw[gp path] (2.027,8.307)--(2.027,1.184)--(11.836,1.184)--(11.836,8.307)--cycle;
\gpdefrectangularnode{gp plot 1}{\pgfpoint{2.027cm}{1.184cm}}{\pgfpoint{11.836cm}{8.307cm}}
\end{tikzpicture}
}
\caption{Relative errors with respect to $M$ for $\log(A)$ with $\kappa(A)=10^4$.}
\label{fig:PointError-1e4}
\scalebox{0.72}{\begin{tikzpicture}[gnuplot]
\tikzset{every node/.append style={font={\ttfamily\fontsize{12.0pt}{14.4pt}\selectfont}}}
\gpmonochromelines
\path (0.000,0.000) rectangle (12.500,8.750);
\gpcolor{color=gp lt color border}
\gpsetlinetype{gp lt border}
\gpsetdashtype{gp dt solid}
\gpsetlinewidth{1.00}
\draw[gp path] (2.027,1.184)--(2.117,1.184);
\draw[gp path] (11.836,1.184)--(11.746,1.184);
\draw[gp path] (2.027,1.629)--(2.207,1.629);
\draw[gp path] (11.836,1.629)--(11.656,1.629);
\node[gp node right] at (1.806,1.629) {$10^{-14}$};
\draw[gp path] (2.027,2.074)--(2.117,2.074);
\draw[gp path] (11.836,2.074)--(11.746,2.074);
\draw[gp path] (2.027,2.520)--(2.207,2.520);
\draw[gp path] (11.836,2.520)--(11.656,2.520);
\node[gp node right] at (1.806,2.520) {$10^{-12}$};
\draw[gp path] (2.027,2.965)--(2.117,2.965);
\draw[gp path] (11.836,2.965)--(11.746,2.965);
\draw[gp path] (2.027,3.410)--(2.207,3.410);
\draw[gp path] (11.836,3.410)--(11.656,3.410);
\node[gp node right] at (1.806,3.410) {$10^{-10}$};
\draw[gp path] (2.027,3.855)--(2.117,3.855);
\draw[gp path] (11.836,3.855)--(11.746,3.855);
\draw[gp path] (2.027,4.300)--(2.207,4.300);
\draw[gp path] (11.836,4.300)--(11.656,4.300);
\node[gp node right] at (1.806,4.300) {$10^{-8}$};
\draw[gp path] (2.027,4.746)--(2.117,4.746);
\draw[gp path] (11.836,4.746)--(11.746,4.746);
\draw[gp path] (2.027,5.191)--(2.207,5.191);
\draw[gp path] (11.836,5.191)--(11.656,5.191);
\node[gp node right] at (1.806,5.191) {$10^{-6}$};
\draw[gp path] (2.027,5.636)--(2.117,5.636);
\draw[gp path] (11.836,5.636)--(11.746,5.636);
\draw[gp path] (2.027,6.081)--(2.207,6.081);
\draw[gp path] (11.836,6.081)--(11.656,6.081);
\node[gp node right] at (1.806,6.081) {$10^{-4}$};
\draw[gp path] (2.027,6.526)--(2.117,6.526);
\draw[gp path] (11.836,6.526)--(11.746,6.526);
\draw[gp path] (2.027,6.971)--(2.207,6.971);
\draw[gp path] (11.836,6.971)--(11.656,6.971);
\node[gp node right] at (1.806,6.971) {$10^{-2}$};
\draw[gp path] (2.027,7.417)--(2.117,7.417);
\draw[gp path] (11.836,7.417)--(11.746,7.417);
\draw[gp path] (2.027,7.862)--(2.207,7.862);
\draw[gp path] (11.836,7.862)--(11.656,7.862);
\node[gp node right] at (1.806,7.862) {$10^{0}$};
\draw[gp path] (2.027,8.307)--(2.117,8.307);
\draw[gp path] (11.836,8.307)--(11.746,8.307);
\draw[gp path] (2.027,1.184)--(2.027,1.364);
\draw[gp path] (2.027,8.307)--(2.027,8.127);
\node[gp node center] at (2.027,0.814) {$0$};
\draw[gp path] (3.428,1.184)--(3.428,1.364);
\draw[gp path] (3.428,8.307)--(3.428,8.127);
\node[gp node center] at (3.428,0.814) {$50$};
\draw[gp path] (4.830,1.184)--(4.830,1.364);
\draw[gp path] (4.830,8.307)--(4.830,8.127);
\node[gp node center] at (4.830,0.814) {$100$};
\draw[gp path] (6.231,1.184)--(6.231,1.364);
\draw[gp path] (6.231,8.307)--(6.231,8.127);
\node[gp node center] at (6.231,0.814) {$150$};
\draw[gp path] (7.632,1.184)--(7.632,1.364);
\draw[gp path] (7.632,8.307)--(7.632,8.127);
\node[gp node center] at (7.632,0.814) {$200$};
\draw[gp path] (9.033,1.184)--(9.033,1.364);
\draw[gp path] (9.033,8.307)--(9.033,8.127);
\node[gp node center] at (9.033,0.814) {$250$};
\draw[gp path] (10.435,1.184)--(10.435,1.364);
\draw[gp path] (10.435,8.307)--(10.435,8.127);
\node[gp node center] at (10.435,0.814) {$300$};
\draw[gp path] (11.836,1.184)--(11.836,1.364);
\draw[gp path] (11.836,8.307)--(11.836,8.127);
\node[gp node center] at (11.836,0.814) {$350$};
\draw[gp path] (2.027,8.307)--(2.027,1.184)--(11.836,1.184)--(11.836,8.307)--cycle;
\node[gp node center,rotate=-270] at (0.295,4.745) {Relative error (2-norm)};
\node[gp node center] at (6.931,0.259) {Number of quadrature nodes ($M$)};
\node[gp node right] at (10.109,7.902) {Algorithm 1 ($m=50$)};
\draw[gp path] (10.330,7.902)--(11.394,7.902);
\draw[gp path] (2.111,7.695)--(2.167,7.656)--(2.279,7.404)--(2.503,6.949)--(2.952,5.939)%
  --(3.849,3.869)--(5.642,2.961)--(9.230,2.916);
\gpsetpointsize{4.00}
\gppoint{gp mark 8}{(2.111,7.695)}
\gppoint{gp mark 8}{(2.167,7.656)}
\gppoint{gp mark 8}{(2.279,7.404)}
\gppoint{gp mark 8}{(2.503,6.949)}
\gppoint{gp mark 8}{(2.952,5.939)}
\gppoint{gp mark 8}{(3.849,3.869)}
\gppoint{gp mark 8}{(5.642,2.961)}
\gppoint{gp mark 8}{(9.230,2.916)}
\gppoint{gp mark 8}{(10.862,7.902)}
\node[gp node right] at (10.109,7.452) {Algorithm 2 ($m=50$)};
\draw[gp path] (10.330,7.452)--(11.394,7.452);
\draw[gp path] (2.167,7.508)--(2.391,7.014)--(2.784,6.180)--(3.849,4.505)--(6.203,2.937)%
  --(11.584,2.900);
\gppoint{gp mark 6}{(2.167,7.508)}
\gppoint{gp mark 6}{(2.391,7.014)}
\gppoint{gp mark 6}{(2.784,6.180)}
\gppoint{gp mark 6}{(3.849,4.505)}
\gppoint{gp mark 6}{(6.203,2.937)}
\gppoint{gp mark 6}{(11.584,2.900)}
\gppoint{gp mark 6}{(10.862,7.452)}
\node[gp node right] at (10.109,7.002) {Algorithm 3 ($m=50$)};
\draw[gp path] (10.330,7.002)--(11.394,7.002);
\gppoint{gp mark 4}{(5.923,2.938)}
\gppoint{gp mark 4}{(10.862,7.002)}
\node[gp node right] at (10.109,6.552) {Algorithm 1 ($m=1600$)};
\draw[gp path] (10.330,6.552)--(11.394,6.552);
\draw[gp path] (2.111,7.695)--(2.167,7.656)--(2.279,7.404)--(2.503,6.949)--(2.952,5.939)%
  --(3.849,3.908)--(5.642,3.163)--(9.230,3.162);
\gppoint{gp mark 9}{(2.111,7.695)}
\gppoint{gp mark 9}{(2.167,7.656)}
\gppoint{gp mark 9}{(2.279,7.404)}
\gppoint{gp mark 9}{(2.503,6.949)}
\gppoint{gp mark 9}{(2.952,5.939)}
\gppoint{gp mark 9}{(3.849,3.908)}
\gppoint{gp mark 9}{(5.642,3.163)}
\gppoint{gp mark 9}{(9.230,3.162)}
\gppoint{gp mark 9}{(10.862,6.552)}
\node[gp node right] at (10.109,6.102) {Algorithm 2 ($m=1600$)};
\draw[gp path] (10.330,6.102)--(11.394,6.102);
\draw[gp path] (2.167,7.508)--(2.391,7.015)--(2.784,6.182)--(3.849,4.511)--(6.203,3.169)%
  --(11.584,3.164);
\gppoint{gp mark 7}{(2.167,7.508)}
\gppoint{gp mark 7}{(2.391,7.015)}
\gppoint{gp mark 7}{(2.784,6.182)}
\gppoint{gp mark 7}{(3.849,4.511)}
\gppoint{gp mark 7}{(6.203,3.169)}
\gppoint{gp mark 7}{(11.584,3.164)}
\gppoint{gp mark 7}{(10.862,6.102)}
\node[gp node right] at (10.109,5.652) {Algorithm 3 ($m=1600$)};
\draw[gp path] (10.330,5.652)--(11.394,5.652);
\gppoint{gp mark 5}{(5.923,3.168)}
\gppoint{gp mark 5}{(10.862,5.652)}
\node[gp node right] at (10.109,5.202) {Error bound for Algorithm 3};
\gpsetdashtype{gp dt 2}
\draw[gp path] (10.330,5.202)--(11.394,5.202);
\draw[gp path] (2.200,8.307)--(2.391,7.941)--(2.559,7.534)--(2.784,7.128)--(3.064,6.723)%
  --(3.288,6.318)--(3.568,5.913)--(3.849,5.509)--(4.129,5.104)--(4.409,4.699)--(4.689,4.294)%
  --(4.970,3.889)--(5.306,3.484)--(5.586,3.079)--(5.923,2.674);
\gpsetdashtype{gp dt solid}
\draw[gp path] (2.027,8.307)--(2.027,1.184)--(11.836,1.184)--(11.836,8.307)--cycle;
\gpdefrectangularnode{gp plot 1}{\pgfpoint{2.027cm}{1.184cm}}{\pgfpoint{11.836cm}{8.307cm}}
\end{tikzpicture}
}
\caption{Relative errors with respect to $M$ for $\log(A)$ with $\kappa(A)=10^7$.}
\label{fig:PointError-1e7}
\end{figure}

Figures~\ref{fig:PointError-1e1}--\ref{fig:PointError-1e7} plot the relative error against the number $M$ of quadrature nodes
for matrix dimensions $m=50$ and $m=1600$.  As $M$ increases,
Algorithm~\ref{alg:adaptive-Tatsuoka} exhibits a slightly faster
convergence than Algorithm~\ref{alg:new-adaptive}.  This trend is qualitatively consistent with the convergence analyses:
the discretization error of Algorithm~\ref{alg:adaptive-Tatsuoka}
decreases as $\OO(\exp(-cM))$, whereas the errors of
Algorithms~\ref{alg:new-adaptive} and~\ref{alg:new-adaptive-error-bound}
are $\OO(\exp(-cn))\approx\OO(\exp(-cM/\log M))$ (note that
$M=2N+1=\OO(n\log n)$). The latter rate is predicted for
Algorithm~\ref{alg:new-adaptive-error-bound} by the analysis, but is
not observable in these plots because this algorithm determines the
number of quadrature nodes in advance from the error bound and therefore
contributes only one plotted point. We note that for Algorithm~\ref{alg:adaptive-Tatsuoka}, this
exponential rate ($\OO(\exp(-cM))$) describes the discretization error; its total error
also includes a truncation-error contribution that does not decrease
with $M$.
For the two plotted dimensions, $m=50$ and $m=1600$, the results
almost coincide.  This observation is consistent with the error
analyses, whose bounds do not depend explicitly on the matrix dimension $m$.

The same figures also show that the error bound for
Algorithm~\ref{alg:new-adaptive-error-bound} is effective for
$\kappa(A)=10^1$ and $10^4$.  For $\kappa(A)=10^7$, however, the
observed error at its sole plotted point, $M=139$, is slightly larger
than the reported bound.  The bound
controls the quadrature error and does not account for floating-point
rounding; this discrepancy may therefore be due to roundoff. Accounting
for finite-precision effects in the error estimate is left for future work.

\subsection{Accuracy versus computation time}
\label{subsec:numerical-time-error}

\begin{figure}[htpb]
\centering
\scalebox{0.72}{\begin{tikzpicture}[gnuplot]
\tikzset{every node/.append style={font={\ttfamily\fontsize{12.0pt}{14.4pt}\selectfont}}}
\gpmonochromelines
\path (0.000,0.000) rectangle (12.500,8.750);
\gpcolor{color=gp lt color border}
\gpsetlinetype{gp lt border}
\gpsetdashtype{gp dt solid}
\gpsetlinewidth{1.00}
\draw[gp path] (2.027,1.184)--(2.117,1.184);
\draw[gp path] (11.836,1.184)--(11.746,1.184);
\draw[gp path] (2.027,1.629)--(2.207,1.629);
\draw[gp path] (11.836,1.629)--(11.656,1.629);
\node[gp node right] at (1.806,1.629) {$10^{-14}$};
\draw[gp path] (2.027,2.074)--(2.117,2.074);
\draw[gp path] (11.836,2.074)--(11.746,2.074);
\draw[gp path] (2.027,2.520)--(2.207,2.520);
\draw[gp path] (11.836,2.520)--(11.656,2.520);
\node[gp node right] at (1.806,2.520) {$10^{-12}$};
\draw[gp path] (2.027,2.965)--(2.117,2.965);
\draw[gp path] (11.836,2.965)--(11.746,2.965);
\draw[gp path] (2.027,3.410)--(2.207,3.410);
\draw[gp path] (11.836,3.410)--(11.656,3.410);
\node[gp node right] at (1.806,3.410) {$10^{-10}$};
\draw[gp path] (2.027,3.855)--(2.117,3.855);
\draw[gp path] (11.836,3.855)--(11.746,3.855);
\draw[gp path] (2.027,4.300)--(2.207,4.300);
\draw[gp path] (11.836,4.300)--(11.656,4.300);
\node[gp node right] at (1.806,4.300) {$10^{-8}$};
\draw[gp path] (2.027,4.746)--(2.117,4.746);
\draw[gp path] (11.836,4.746)--(11.746,4.746);
\draw[gp path] (2.027,5.191)--(2.207,5.191);
\draw[gp path] (11.836,5.191)--(11.656,5.191);
\node[gp node right] at (1.806,5.191) {$10^{-6}$};
\draw[gp path] (2.027,5.636)--(2.117,5.636);
\draw[gp path] (11.836,5.636)--(11.746,5.636);
\draw[gp path] (2.027,6.081)--(2.207,6.081);
\draw[gp path] (11.836,6.081)--(11.656,6.081);
\node[gp node right] at (1.806,6.081) {$10^{-4}$};
\draw[gp path] (2.027,6.526)--(2.117,6.526);
\draw[gp path] (11.836,6.526)--(11.746,6.526);
\draw[gp path] (2.027,6.971)--(2.207,6.971);
\draw[gp path] (11.836,6.971)--(11.656,6.971);
\node[gp node right] at (1.806,6.971) {$10^{-2}$};
\draw[gp path] (2.027,7.417)--(2.117,7.417);
\draw[gp path] (11.836,7.417)--(11.746,7.417);
\draw[gp path] (2.027,7.862)--(2.207,7.862);
\draw[gp path] (11.836,7.862)--(11.656,7.862);
\node[gp node right] at (1.806,7.862) {$10^{0}$};
\draw[gp path] (2.027,8.307)--(2.117,8.307);
\draw[gp path] (11.836,8.307)--(11.746,8.307);
\draw[gp path] (2.027,1.184)--(2.027,1.364);
\draw[gp path] (2.027,8.307)--(2.027,8.127);
\node[gp node center] at (2.027,0.814) {$0$};
\draw[gp path] (3.428,1.184)--(3.428,1.364);
\draw[gp path] (3.428,8.307)--(3.428,8.127);
\node[gp node center] at (3.428,0.814) {$1$};
\draw[gp path] (4.830,1.184)--(4.830,1.364);
\draw[gp path] (4.830,8.307)--(4.830,8.127);
\node[gp node center] at (4.830,0.814) {$2$};
\draw[gp path] (6.231,1.184)--(6.231,1.364);
\draw[gp path] (6.231,8.307)--(6.231,8.127);
\node[gp node center] at (6.231,0.814) {$3$};
\draw[gp path] (7.632,1.184)--(7.632,1.364);
\draw[gp path] (7.632,8.307)--(7.632,8.127);
\node[gp node center] at (7.632,0.814) {$4$};
\draw[gp path] (9.033,1.184)--(9.033,1.364);
\draw[gp path] (9.033,8.307)--(9.033,8.127);
\node[gp node center] at (9.033,0.814) {$5$};
\draw[gp path] (10.435,1.184)--(10.435,1.364);
\draw[gp path] (10.435,8.307)--(10.435,8.127);
\node[gp node center] at (10.435,0.814) {$6$};
\draw[gp path] (11.836,1.184)--(11.836,1.364);
\draw[gp path] (11.836,8.307)--(11.836,8.127);
\node[gp node center] at (11.836,0.814) {$7$};
\draw[gp path] (2.027,8.307)--(2.027,1.184)--(11.836,1.184)--(11.836,8.307)--cycle;
\node[gp node center,rotate=-270] at (0.295,4.745) {Relative error (2-norm)};
\node[gp node center] at (6.931,0.259) {Computation time [s]};
\node[gp node right] at (10.109,7.902) {Algorithm 1 ($m=400$)};
\draw[gp path] (10.330,7.902)--(11.394,7.902);
\draw[gp path] (2.420,7.524)--(2.488,6.832)--(2.574,5.232)--(2.723,2.356)--(3.074,1.529)%
  --(3.028,1.529)--(3.047,1.529)--(3.040,1.529);
\gpsetpointsize{4.00}
\gppoint{gp mark 8}{(2.420,7.524)}
\gppoint{gp mark 8}{(2.488,6.832)}
\gppoint{gp mark 8}{(2.574,5.232)}
\gppoint{gp mark 8}{(2.723,2.356)}
\gppoint{gp mark 8}{(3.074,1.529)}
\gppoint{gp mark 8}{(3.028,1.529)}
\gppoint{gp mark 8}{(3.047,1.529)}
\gppoint{gp mark 8}{(3.040,1.529)}
\gppoint{gp mark 8}{(10.862,7.902)}
\node[gp node right] at (10.109,7.452) {Algorithm 2 ($m=400$)};
\draw[gp path] (10.330,7.452)--(11.394,7.452);
\draw[gp path] (2.066,7.065)--(2.065,7.065)--(2.128,5.930)--(2.233,3.354)--(2.526,1.528)%
  --(3.214,1.528)--(3.489,1.528)--(3.513,1.528);
\gppoint{gp mark 6}{(2.066,7.065)}
\gppoint{gp mark 6}{(2.065,7.065)}
\gppoint{gp mark 6}{(2.128,5.930)}
\gppoint{gp mark 6}{(2.233,3.354)}
\gppoint{gp mark 6}{(2.526,1.528)}
\gppoint{gp mark 6}{(3.214,1.528)}
\gppoint{gp mark 6}{(3.489,1.528)}
\gppoint{gp mark 6}{(3.513,1.528)}
\gppoint{gp mark 6}{(10.862,7.452)}
\node[gp node right] at (10.109,7.002) {Algorithm 3 ($m=400$)};
\draw[gp path] (10.330,7.002)--(11.394,7.002);
\draw[gp path] (2.177,7.065)--(2.167,7.065)--(2.225,5.930)--(2.294,4.676)--(2.561,1.528)%
  --(2.600,1.528)--(2.572,1.528)--(2.566,1.528);
\gppoint{gp mark 4}{(2.177,7.065)}
\gppoint{gp mark 4}{(2.167,7.065)}
\gppoint{gp mark 4}{(2.225,5.930)}
\gppoint{gp mark 4}{(2.294,4.676)}
\gppoint{gp mark 4}{(2.561,1.528)}
\gppoint{gp mark 4}{(2.600,1.528)}
\gppoint{gp mark 4}{(2.572,1.528)}
\gppoint{gp mark 4}{(2.566,1.528)}
\gppoint{gp mark 4}{(10.862,7.002)}
\node[gp node right] at (10.109,6.552) {Algorithm 1 ($m=800$)};
\draw[gp path] (10.330,6.552)--(11.394,6.552);
\draw[gp path] (3.731,7.524)--(3.956,6.832)--(4.791,5.232)--(5.035,2.356)--(6.087,1.555)%
  --(6.064,1.555)--(6.100,1.555)--(5.849,1.555);
\gppoint{gp mark 9}{(3.731,7.524)}
\gppoint{gp mark 9}{(3.956,6.832)}
\gppoint{gp mark 9}{(4.791,5.232)}
\gppoint{gp mark 9}{(5.035,2.356)}
\gppoint{gp mark 9}{(6.087,1.555)}
\gppoint{gp mark 9}{(6.064,1.555)}
\gppoint{gp mark 9}{(6.100,1.555)}
\gppoint{gp mark 9}{(5.849,1.555)}
\gppoint{gp mark 9}{(10.862,6.552)}
\node[gp node right] at (10.109,6.102) {Algorithm 2 ($m=800$)};
\draw[gp path] (10.330,6.102)--(11.394,6.102);
\draw[gp path] (2.159,7.065)--(2.161,7.065)--(2.385,5.930)--(2.777,3.354)--(3.788,1.555)%
  --(6.033,1.555)--(6.289,1.555)--(6.238,1.555);
\gppoint{gp mark 7}{(2.159,7.065)}
\gppoint{gp mark 7}{(2.161,7.065)}
\gppoint{gp mark 7}{(2.385,5.930)}
\gppoint{gp mark 7}{(2.777,3.354)}
\gppoint{gp mark 7}{(3.788,1.555)}
\gppoint{gp mark 7}{(6.033,1.555)}
\gppoint{gp mark 7}{(6.289,1.555)}
\gppoint{gp mark 7}{(6.238,1.555)}
\gppoint{gp mark 7}{(10.862,6.102)}
\node[gp node right] at (10.109,5.652) {Algorithm 3 ($m=800$)};
\draw[gp path] (10.330,5.652)--(11.394,5.652);
\draw[gp path] (2.458,7.065)--(2.469,7.065)--(2.669,5.930)--(2.829,4.676)--(3.527,1.546)%
  --(3.457,1.546)--(3.560,1.546)--(3.520,1.546);
\gppoint{gp mark 5}{(2.458,7.065)}
\gppoint{gp mark 5}{(2.469,7.065)}
\gppoint{gp mark 5}{(2.669,5.930)}
\gppoint{gp mark 5}{(2.829,4.676)}
\gppoint{gp mark 5}{(3.527,1.546)}
\gppoint{gp mark 5}{(3.457,1.546)}
\gppoint{gp mark 5}{(3.560,1.546)}
\gppoint{gp mark 5}{(3.520,1.546)}
\gppoint{gp mark 5}{(10.862,5.652)}
\draw[gp path] (2.027,8.307)--(2.027,1.184)--(11.836,1.184)--(11.836,8.307)--cycle;
\gpdefrectangularnode{gp plot 1}{\pgfpoint{2.027cm}{1.184cm}}{\pgfpoint{11.836cm}{8.307cm}}
\end{tikzpicture}
}
\caption{Relative errors with respect to computation time [s] for $\log(A)$ with $\kappa(A)=10^1$.}
\label{fig:TimeError-1e1}
\scalebox{0.72}{\begin{tikzpicture}[gnuplot]
\tikzset{every node/.append style={font={\ttfamily\fontsize{12.0pt}{14.4pt}\selectfont}}}
\gpmonochromelines
\path (0.000,0.000) rectangle (12.500,8.750);
\gpcolor{color=gp lt color border}
\gpsetlinetype{gp lt border}
\gpsetdashtype{gp dt solid}
\gpsetlinewidth{1.00}
\draw[gp path] (2.027,1.184)--(2.117,1.184);
\draw[gp path] (11.836,1.184)--(11.746,1.184);
\draw[gp path] (2.027,1.629)--(2.207,1.629);
\draw[gp path] (11.836,1.629)--(11.656,1.629);
\node[gp node right] at (1.806,1.629) {$10^{-14}$};
\draw[gp path] (2.027,2.074)--(2.117,2.074);
\draw[gp path] (11.836,2.074)--(11.746,2.074);
\draw[gp path] (2.027,2.520)--(2.207,2.520);
\draw[gp path] (11.836,2.520)--(11.656,2.520);
\node[gp node right] at (1.806,2.520) {$10^{-12}$};
\draw[gp path] (2.027,2.965)--(2.117,2.965);
\draw[gp path] (11.836,2.965)--(11.746,2.965);
\draw[gp path] (2.027,3.410)--(2.207,3.410);
\draw[gp path] (11.836,3.410)--(11.656,3.410);
\node[gp node right] at (1.806,3.410) {$10^{-10}$};
\draw[gp path] (2.027,3.855)--(2.117,3.855);
\draw[gp path] (11.836,3.855)--(11.746,3.855);
\draw[gp path] (2.027,4.300)--(2.207,4.300);
\draw[gp path] (11.836,4.300)--(11.656,4.300);
\node[gp node right] at (1.806,4.300) {$10^{-8}$};
\draw[gp path] (2.027,4.746)--(2.117,4.746);
\draw[gp path] (11.836,4.746)--(11.746,4.746);
\draw[gp path] (2.027,5.191)--(2.207,5.191);
\draw[gp path] (11.836,5.191)--(11.656,5.191);
\node[gp node right] at (1.806,5.191) {$10^{-6}$};
\draw[gp path] (2.027,5.636)--(2.117,5.636);
\draw[gp path] (11.836,5.636)--(11.746,5.636);
\draw[gp path] (2.027,6.081)--(2.207,6.081);
\draw[gp path] (11.836,6.081)--(11.656,6.081);
\node[gp node right] at (1.806,6.081) {$10^{-4}$};
\draw[gp path] (2.027,6.526)--(2.117,6.526);
\draw[gp path] (11.836,6.526)--(11.746,6.526);
\draw[gp path] (2.027,6.971)--(2.207,6.971);
\draw[gp path] (11.836,6.971)--(11.656,6.971);
\node[gp node right] at (1.806,6.971) {$10^{-2}$};
\draw[gp path] (2.027,7.417)--(2.117,7.417);
\draw[gp path] (11.836,7.417)--(11.746,7.417);
\draw[gp path] (2.027,7.862)--(2.207,7.862);
\draw[gp path] (11.836,7.862)--(11.656,7.862);
\node[gp node right] at (1.806,7.862) {$10^{0}$};
\draw[gp path] (2.027,8.307)--(2.117,8.307);
\draw[gp path] (11.836,8.307)--(11.746,8.307);
\draw[gp path] (2.027,1.184)--(2.027,1.364);
\draw[gp path] (2.027,8.307)--(2.027,8.127);
\node[gp node center] at (2.027,0.814) {$0$};
\draw[gp path] (3.428,1.184)--(3.428,1.364);
\draw[gp path] (3.428,8.307)--(3.428,8.127);
\node[gp node center] at (3.428,0.814) {$1$};
\draw[gp path] (4.830,1.184)--(4.830,1.364);
\draw[gp path] (4.830,8.307)--(4.830,8.127);
\node[gp node center] at (4.830,0.814) {$2$};
\draw[gp path] (6.231,1.184)--(6.231,1.364);
\draw[gp path] (6.231,8.307)--(6.231,8.127);
\node[gp node center] at (6.231,0.814) {$3$};
\draw[gp path] (7.632,1.184)--(7.632,1.364);
\draw[gp path] (7.632,8.307)--(7.632,8.127);
\node[gp node center] at (7.632,0.814) {$4$};
\draw[gp path] (9.033,1.184)--(9.033,1.364);
\draw[gp path] (9.033,8.307)--(9.033,8.127);
\node[gp node center] at (9.033,0.814) {$5$};
\draw[gp path] (10.435,1.184)--(10.435,1.364);
\draw[gp path] (10.435,8.307)--(10.435,8.127);
\node[gp node center] at (10.435,0.814) {$6$};
\draw[gp path] (11.836,1.184)--(11.836,1.364);
\draw[gp path] (11.836,8.307)--(11.836,8.127);
\node[gp node center] at (11.836,0.814) {$7$};
\draw[gp path] (2.027,8.307)--(2.027,1.184)--(11.836,1.184)--(11.836,8.307)--cycle;
\node[gp node center,rotate=-270] at (0.295,4.745) {Relative error (2-norm)};
\node[gp node center] at (6.931,0.259) {Computation time [s]};
\node[gp node right] at (10.109,7.902) {Algorithm 1 ($m=400$)};
\draw[gp path] (10.330,7.902)--(11.394,7.902);
\draw[gp path] (2.543,7.652)--(2.649,7.216)--(2.658,6.496)--(2.863,4.807)--(3.064,1.999)%
  --(3.580,1.978)--(3.610,1.978)--(3.870,1.978);
\gpsetpointsize{4.00}
\gppoint{gp mark 8}{(2.543,7.652)}
\gppoint{gp mark 8}{(2.649,7.216)}
\gppoint{gp mark 8}{(2.658,6.496)}
\gppoint{gp mark 8}{(2.863,4.807)}
\gppoint{gp mark 8}{(3.064,1.999)}
\gppoint{gp mark 8}{(3.580,1.978)}
\gppoint{gp mark 8}{(3.610,1.978)}
\gppoint{gp mark 8}{(3.870,1.978)}
\gppoint{gp mark 8}{(10.862,7.902)}
\node[gp node right] at (10.109,7.452) {Algorithm 2 ($m=400$)};
\draw[gp path] (10.330,7.452)--(11.394,7.452);
\draw[gp path] (2.074,7.392)--(2.072,7.392)--(2.145,6.665)--(2.295,5.465)--(2.612,2.816)%
  --(3.464,1.974)--(3.475,1.974)--(5.249,1.976);
\gppoint{gp mark 6}{(2.074,7.392)}
\gppoint{gp mark 6}{(2.072,7.392)}
\gppoint{gp mark 6}{(2.145,6.665)}
\gppoint{gp mark 6}{(2.295,5.465)}
\gppoint{gp mark 6}{(2.612,2.816)}
\gppoint{gp mark 6}{(3.464,1.974)}
\gppoint{gp mark 6}{(3.475,1.974)}
\gppoint{gp mark 6}{(5.249,1.976)}
\gppoint{gp mark 6}{(10.862,7.452)}
\node[gp node right] at (10.109,7.002) {Algorithm 3 ($m=400$)};
\draw[gp path] (10.330,7.002)--(11.394,7.002);
\draw[gp path] (2.152,7.392)--(2.148,7.392)--(2.223,6.665)--(2.297,6.047)--(2.586,3.445)%
  --(2.828,1.974)--(2.835,1.974)--(2.842,1.974);
\gppoint{gp mark 4}{(2.152,7.392)}
\gppoint{gp mark 4}{(2.148,7.392)}
\gppoint{gp mark 4}{(2.223,6.665)}
\gppoint{gp mark 4}{(2.297,6.047)}
\gppoint{gp mark 4}{(2.586,3.445)}
\gppoint{gp mark 4}{(2.828,1.974)}
\gppoint{gp mark 4}{(2.835,1.974)}
\gppoint{gp mark 4}{(2.842,1.974)}
\gppoint{gp mark 4}{(10.862,7.002)}
\node[gp node right] at (10.109,6.552) {Algorithm 1 ($m=800$)};
\draw[gp path] (10.330,6.552)--(11.394,6.552);
\draw[gp path] (3.668,7.652)--(3.881,7.216)--(4.452,6.496)--(5.070,4.807)--(5.940,2.040)%
  --(8.122,2.031)--(7.566,2.031)--(7.507,2.031);
\gppoint{gp mark 9}{(3.668,7.652)}
\gppoint{gp mark 9}{(3.881,7.216)}
\gppoint{gp mark 9}{(4.452,6.496)}
\gppoint{gp mark 9}{(5.070,4.807)}
\gppoint{gp mark 9}{(5.940,2.040)}
\gppoint{gp mark 9}{(8.122,2.031)}
\gppoint{gp mark 9}{(7.566,2.031)}
\gppoint{gp mark 9}{(7.507,2.031)}
\gppoint{gp mark 9}{(10.862,6.552)}
\node[gp node right] at (10.109,6.102) {Algorithm 2 ($m=800$)};
\draw[gp path] (10.330,6.102)--(11.394,6.102);
\draw[gp path] (2.161,7.392)--(2.162,7.392)--(2.374,6.665)--(2.751,5.465)--(3.860,2.816)%
  --(6.243,2.036)--(5.993,2.036)--(11.228,2.033);
\gppoint{gp mark 7}{(2.161,7.392)}
\gppoint{gp mark 7}{(2.162,7.392)}
\gppoint{gp mark 7}{(2.374,6.665)}
\gppoint{gp mark 7}{(2.751,5.465)}
\gppoint{gp mark 7}{(3.860,2.816)}
\gppoint{gp mark 7}{(6.243,2.036)}
\gppoint{gp mark 7}{(5.993,2.036)}
\gppoint{gp mark 7}{(11.228,2.033)}
\gppoint{gp mark 7}{(10.862,6.102)}
\node[gp node right] at (10.109,5.652) {Algorithm 3 ($m=800$)};
\draw[gp path] (10.330,5.652)--(11.394,5.652);
\draw[gp path] (2.438,7.392)--(2.431,7.392)--(2.637,6.665)--(2.787,6.048)--(3.758,3.445)%
  --(4.784,2.038)--(4.922,2.038)--(4.861,2.038);
\gppoint{gp mark 5}{(2.438,7.392)}
\gppoint{gp mark 5}{(2.431,7.392)}
\gppoint{gp mark 5}{(2.637,6.665)}
\gppoint{gp mark 5}{(2.787,6.048)}
\gppoint{gp mark 5}{(3.758,3.445)}
\gppoint{gp mark 5}{(4.784,2.038)}
\gppoint{gp mark 5}{(4.922,2.038)}
\gppoint{gp mark 5}{(4.861,2.038)}
\gppoint{gp mark 5}{(10.862,5.652)}
\draw[gp path] (2.027,8.307)--(2.027,1.184)--(11.836,1.184)--(11.836,8.307)--cycle;
\gpdefrectangularnode{gp plot 1}{\pgfpoint{2.027cm}{1.184cm}}{\pgfpoint{11.836cm}{8.307cm}}
\end{tikzpicture}
}
\caption{Relative errors with respect to computation time [s] for $\log(A)$ with $\kappa(A)=10^4$.}
\label{fig:TimeError-1e4}
\scalebox{0.72}{\begin{tikzpicture}[gnuplot]
\tikzset{every node/.append style={font={\ttfamily\fontsize{12.0pt}{14.4pt}\selectfont}}}
\gpmonochromelines
\path (0.000,0.000) rectangle (12.500,8.750);
\gpcolor{color=gp lt color border}
\gpsetlinetype{gp lt border}
\gpsetdashtype{gp dt solid}
\gpsetlinewidth{1.00}
\draw[gp path] (2.027,1.184)--(2.117,1.184);
\draw[gp path] (11.836,1.184)--(11.746,1.184);
\draw[gp path] (2.027,1.629)--(2.207,1.629);
\draw[gp path] (11.836,1.629)--(11.656,1.629);
\node[gp node right] at (1.806,1.629) {$10^{-14}$};
\draw[gp path] (2.027,2.074)--(2.117,2.074);
\draw[gp path] (11.836,2.074)--(11.746,2.074);
\draw[gp path] (2.027,2.520)--(2.207,2.520);
\draw[gp path] (11.836,2.520)--(11.656,2.520);
\node[gp node right] at (1.806,2.520) {$10^{-12}$};
\draw[gp path] (2.027,2.965)--(2.117,2.965);
\draw[gp path] (11.836,2.965)--(11.746,2.965);
\draw[gp path] (2.027,3.410)--(2.207,3.410);
\draw[gp path] (11.836,3.410)--(11.656,3.410);
\node[gp node right] at (1.806,3.410) {$10^{-10}$};
\draw[gp path] (2.027,3.855)--(2.117,3.855);
\draw[gp path] (11.836,3.855)--(11.746,3.855);
\draw[gp path] (2.027,4.300)--(2.207,4.300);
\draw[gp path] (11.836,4.300)--(11.656,4.300);
\node[gp node right] at (1.806,4.300) {$10^{-8}$};
\draw[gp path] (2.027,4.746)--(2.117,4.746);
\draw[gp path] (11.836,4.746)--(11.746,4.746);
\draw[gp path] (2.027,5.191)--(2.207,5.191);
\draw[gp path] (11.836,5.191)--(11.656,5.191);
\node[gp node right] at (1.806,5.191) {$10^{-6}$};
\draw[gp path] (2.027,5.636)--(2.117,5.636);
\draw[gp path] (11.836,5.636)--(11.746,5.636);
\draw[gp path] (2.027,6.081)--(2.207,6.081);
\draw[gp path] (11.836,6.081)--(11.656,6.081);
\node[gp node right] at (1.806,6.081) {$10^{-4}$};
\draw[gp path] (2.027,6.526)--(2.117,6.526);
\draw[gp path] (11.836,6.526)--(11.746,6.526);
\draw[gp path] (2.027,6.971)--(2.207,6.971);
\draw[gp path] (11.836,6.971)--(11.656,6.971);
\node[gp node right] at (1.806,6.971) {$10^{-2}$};
\draw[gp path] (2.027,7.417)--(2.117,7.417);
\draw[gp path] (11.836,7.417)--(11.746,7.417);
\draw[gp path] (2.027,7.862)--(2.207,7.862);
\draw[gp path] (11.836,7.862)--(11.656,7.862);
\node[gp node right] at (1.806,7.862) {$10^{0}$};
\draw[gp path] (2.027,8.307)--(2.117,8.307);
\draw[gp path] (11.836,8.307)--(11.746,8.307);
\draw[gp path] (2.027,1.184)--(2.027,1.364);
\draw[gp path] (2.027,8.307)--(2.027,8.127);
\node[gp node center] at (2.027,0.814) {$0$};
\draw[gp path] (3.428,1.184)--(3.428,1.364);
\draw[gp path] (3.428,8.307)--(3.428,8.127);
\node[gp node center] at (3.428,0.814) {$1$};
\draw[gp path] (4.830,1.184)--(4.830,1.364);
\draw[gp path] (4.830,8.307)--(4.830,8.127);
\node[gp node center] at (4.830,0.814) {$2$};
\draw[gp path] (6.231,1.184)--(6.231,1.364);
\draw[gp path] (6.231,8.307)--(6.231,8.127);
\node[gp node center] at (6.231,0.814) {$3$};
\draw[gp path] (7.632,1.184)--(7.632,1.364);
\draw[gp path] (7.632,8.307)--(7.632,8.127);
\node[gp node center] at (7.632,0.814) {$4$};
\draw[gp path] (9.033,1.184)--(9.033,1.364);
\draw[gp path] (9.033,8.307)--(9.033,8.127);
\node[gp node center] at (9.033,0.814) {$5$};
\draw[gp path] (10.435,1.184)--(10.435,1.364);
\draw[gp path] (10.435,8.307)--(10.435,8.127);
\node[gp node center] at (10.435,0.814) {$6$};
\draw[gp path] (11.836,1.184)--(11.836,1.364);
\draw[gp path] (11.836,8.307)--(11.836,8.127);
\node[gp node center] at (11.836,0.814) {$7$};
\draw[gp path] (2.027,8.307)--(2.027,1.184)--(11.836,1.184)--(11.836,8.307)--cycle;
\node[gp node center,rotate=-270] at (0.295,4.745) {Relative error (2-norm)};
\node[gp node center] at (6.931,0.259) {Computation time [s]};
\node[gp node right] at (10.109,7.902) {Algorithm 1 ($m=400$)};
\draw[gp path] (10.330,7.902)--(11.394,7.902);
\draw[gp path] (2.432,7.656)--(2.518,7.404)--(2.607,6.949)--(2.773,5.939)--(3.072,3.908)%
  --(3.603,3.064)--(4.840,3.055)--(4.752,3.055);
\gpsetpointsize{4.00}
\gppoint{gp mark 8}{(2.432,7.656)}
\gppoint{gp mark 8}{(2.518,7.404)}
\gppoint{gp mark 8}{(2.607,6.949)}
\gppoint{gp mark 8}{(2.773,5.939)}
\gppoint{gp mark 8}{(3.072,3.908)}
\gppoint{gp mark 8}{(3.603,3.064)}
\gppoint{gp mark 8}{(4.840,3.055)}
\gppoint{gp mark 8}{(4.752,3.055)}
\gppoint{gp mark 8}{(10.862,7.902)}
\node[gp node right] at (10.109,7.452) {Algorithm 2 ($m=400$)};
\draw[gp path] (10.330,7.452)--(11.394,7.452);
\draw[gp path] (2.071,7.508)--(2.139,7.015)--(2.265,6.182)--(2.669,4.511)--(3.501,3.061)%
  --(3.483,3.061)--(5.801,3.042);
\gppoint{gp mark 6}{(2.071,7.508)}
\gppoint{gp mark 6}{(2.071,7.508)}
\gppoint{gp mark 6}{(2.139,7.015)}
\gppoint{gp mark 6}{(2.265,6.182)}
\gppoint{gp mark 6}{(2.669,4.511)}
\gppoint{gp mark 6}{(3.501,3.061)}
\gppoint{gp mark 6}{(3.483,3.061)}
\gppoint{gp mark 6}{(5.801,3.042)}
\gppoint{gp mark 6}{(10.862,7.452)}
\node[gp node right] at (10.109,7.002) {Algorithm 3 ($m=400$)};
\draw[gp path] (10.330,7.002)--(11.394,7.002);
\draw[gp path] (2.148,7.508)--(2.154,7.508)--(2.241,7.015)--(2.300,6.655)--(2.595,4.877)%
  --(3.173,3.065)--(3.474,3.076)--(3.549,3.076);
\gppoint{gp mark 4}{(2.148,7.508)}
\gppoint{gp mark 4}{(2.154,7.508)}
\gppoint{gp mark 4}{(2.241,7.015)}
\gppoint{gp mark 4}{(2.300,6.655)}
\gppoint{gp mark 4}{(2.595,4.877)}
\gppoint{gp mark 4}{(3.173,3.065)}
\gppoint{gp mark 4}{(3.474,3.076)}
\gppoint{gp mark 4}{(3.549,3.076)}
\gppoint{gp mark 4}{(10.862,7.002)}
\node[gp node right] at (10.109,6.552) {Algorithm 1 ($m=800$)};
\draw[gp path] (10.330,6.552)--(11.394,6.552);
\draw[gp path] (3.480,7.656)--(3.721,7.404)--(4.155,6.949)--(4.750,5.939)--(5.600,3.908)%
  --(7.435,3.109)--(11.103,3.102)--(10.852,3.102);
\gppoint{gp mark 9}{(3.480,7.656)}
\gppoint{gp mark 9}{(3.721,7.404)}
\gppoint{gp mark 9}{(4.155,6.949)}
\gppoint{gp mark 9}{(4.750,5.939)}
\gppoint{gp mark 9}{(5.600,3.908)}
\gppoint{gp mark 9}{(7.435,3.109)}
\gppoint{gp mark 9}{(11.103,3.102)}
\gppoint{gp mark 9}{(10.852,3.102)}
\gppoint{gp mark 9}{(10.862,6.552)}
\node[gp node right] at (10.109,6.102) {Algorithm 2 ($m=800$)};
\draw[gp path] (10.330,6.102)--(11.394,6.102);
\draw[gp path] (2.162,7.508)--(2.174,7.508)--(2.386,7.015)--(2.729,6.182)--(3.720,4.511)%
  --(6.046,3.104)--(6.012,3.104)--(11.263,3.096);
\gppoint{gp mark 7}{(2.162,7.508)}
\gppoint{gp mark 7}{(2.174,7.508)}
\gppoint{gp mark 7}{(2.386,7.015)}
\gppoint{gp mark 7}{(2.729,6.182)}
\gppoint{gp mark 7}{(3.720,4.511)}
\gppoint{gp mark 7}{(6.046,3.104)}
\gppoint{gp mark 7}{(6.012,3.104)}
\gppoint{gp mark 7}{(11.263,3.096)}
\gppoint{gp mark 7}{(10.862,6.102)}
\node[gp node right] at (10.109,5.652) {Algorithm 3 ($m=800$)};
\draw[gp path] (10.330,5.652)--(11.394,5.652);
\draw[gp path] (2.420,7.508)--(2.417,7.508)--(2.640,7.015)--(2.793,6.655)--(3.729,4.877)%
  --(5.107,3.133)--(5.919,3.123)--(5.929,3.123);
\gppoint{gp mark 5}{(2.420,7.508)}
\gppoint{gp mark 5}{(2.417,7.508)}
\gppoint{gp mark 5}{(2.640,7.015)}
\gppoint{gp mark 5}{(2.793,6.655)}
\gppoint{gp mark 5}{(3.729,4.877)}
\gppoint{gp mark 5}{(5.107,3.133)}
\gppoint{gp mark 5}{(5.919,3.123)}
\gppoint{gp mark 5}{(5.929,3.123)}
\gppoint{gp mark 5}{(10.862,5.652)}
\draw[gp path] (2.027,8.307)--(2.027,1.184)--(11.836,1.184)--(11.836,8.307)--cycle;
\gpdefrectangularnode{gp plot 1}{\pgfpoint{2.027cm}{1.184cm}}{\pgfpoint{11.836cm}{8.307cm}}
\end{tikzpicture}
}
\caption{Relative errors with respect to computation time [s] for $\log(A)$ with $\kappa(A)=10^7$.}
\label{fig:TimeError-1e7}
\end{figure}

Figures~\ref{fig:TimeError-1e1}--\ref{fig:TimeError-1e7} plot the relative error against elapsed computation time
for matrix dimensions $m=400$ and $m=800$.
At the beginning of the time-error plots, Algorithms~\ref{alg:new-adaptive}
and~\ref{alg:new-adaptive-error-bound} attain substantially lower errors
than Algorithm~\ref{alg:adaptive-Tatsuoka} at comparable elapsed times.
Algorithm~\ref{alg:adaptive-Tatsuoka} does not close this gap over the
plotted time range, despite its slightly faster convergence as the
number of quadrature nodes increases (as seen in
Figs.~\ref{fig:PointError-1e1}--\ref{fig:PointError-1e7}).
The gap becomes more pronounced as the matrix dimension increases from
$m=400$ to $m=800$, showing that convergence as a function of the number
of quadrature nodes alone does not determine time-to-accuracy.

\subsection{Termination time}
\label{subsec:numerical-termination-time}

\begin{figure}[htpb]
\centering
\scalebox{0.72}{\begin{tikzpicture}[gnuplot]
\tikzset{every node/.append style={font={\ttfamily\fontsize{12.0pt}{14.4pt}\selectfont}}}
\gpmonochromelines
\path (0.000,0.000) rectangle (12.500,8.750);
\gpcolor{color=gp lt color border}
\gpsetlinetype{gp lt border}
\gpsetdashtype{gp dt solid}
\gpsetlinewidth{1.00}
\draw[gp path] (2.027,1.184)--(2.207,1.184);
\draw[gp path] (11.836,1.184)--(11.656,1.184);
\node[gp node right] at (1.806,1.184) {$0.001$};
\draw[gp path] (2.027,1.613)--(2.117,1.613);
\draw[gp path] (11.836,1.613)--(11.746,1.613);
\draw[gp path] (2.027,2.180)--(2.117,2.180);
\draw[gp path] (11.836,2.180)--(11.746,2.180);
\draw[gp path] (2.027,2.471)--(2.117,2.471);
\draw[gp path] (11.836,2.471)--(11.746,2.471);
\draw[gp path] (2.027,2.609)--(2.207,2.609);
\draw[gp path] (11.836,2.609)--(11.656,2.609);
\node[gp node right] at (1.806,2.609) {$0.01$};
\draw[gp path] (2.027,3.037)--(2.117,3.037);
\draw[gp path] (11.836,3.037)--(11.746,3.037);
\draw[gp path] (2.027,3.604)--(2.117,3.604);
\draw[gp path] (11.836,3.604)--(11.746,3.604);
\draw[gp path] (2.027,3.895)--(2.117,3.895);
\draw[gp path] (11.836,3.895)--(11.746,3.895);
\draw[gp path] (2.027,4.033)--(2.207,4.033);
\draw[gp path] (11.836,4.033)--(11.656,4.033);
\node[gp node right] at (1.806,4.033) {$0.1$};
\draw[gp path] (2.027,4.462)--(2.117,4.462);
\draw[gp path] (11.836,4.462)--(11.746,4.462);
\draw[gp path] (2.027,5.029)--(2.117,5.029);
\draw[gp path] (11.836,5.029)--(11.746,5.029);
\draw[gp path] (2.027,5.320)--(2.117,5.320);
\draw[gp path] (11.836,5.320)--(11.746,5.320);
\draw[gp path] (2.027,5.458)--(2.207,5.458);
\draw[gp path] (11.836,5.458)--(11.656,5.458);
\node[gp node right] at (1.806,5.458) {$1$};
\draw[gp path] (2.027,5.887)--(2.117,5.887);
\draw[gp path] (11.836,5.887)--(11.746,5.887);
\draw[gp path] (2.027,6.454)--(2.117,6.454);
\draw[gp path] (11.836,6.454)--(11.746,6.454);
\draw[gp path] (2.027,6.744)--(2.117,6.744);
\draw[gp path] (11.836,6.744)--(11.746,6.744);
\draw[gp path] (2.027,6.882)--(2.207,6.882);
\draw[gp path] (11.836,6.882)--(11.656,6.882);
\node[gp node right] at (1.806,6.882) {$10$};
\draw[gp path] (2.027,7.311)--(2.117,7.311);
\draw[gp path] (11.836,7.311)--(11.746,7.311);
\draw[gp path] (2.027,7.878)--(2.117,7.878);
\draw[gp path] (11.836,7.878)--(11.746,7.878);
\draw[gp path] (2.027,8.169)--(2.117,8.169);
\draw[gp path] (11.836,8.169)--(11.746,8.169);
\draw[gp path] (2.027,8.307)--(2.207,8.307);
\draw[gp path] (11.836,8.307)--(11.656,8.307);
\node[gp node right] at (1.806,8.307) {$100$};
\draw[gp path] (2.027,1.184)--(2.027,1.364);
\draw[gp path] (2.027,8.307)--(2.027,8.127);
\node[gp node center] at (2.027,0.814) {20};
\draw[gp path] (2.891,1.184)--(2.891,1.274);
\draw[gp path] (2.891,8.307)--(2.891,8.217);
\draw[gp path] (3.503,1.184)--(3.503,1.274);
\draw[gp path] (3.503,8.307)--(3.503,8.217);
\draw[gp path] (3.979,1.184)--(3.979,1.364);
\draw[gp path] (3.979,8.307)--(3.979,8.127);
\node[gp node center] at (3.979,0.814) {50};
\draw[gp path] (4.367,1.184)--(4.367,1.274);
\draw[gp path] (4.367,8.307)--(4.367,8.217);
\draw[gp path] (4.695,1.184)--(4.695,1.274);
\draw[gp path] (4.695,8.307)--(4.695,8.217);
\draw[gp path] (4.980,1.184)--(4.980,1.274);
\draw[gp path] (4.980,8.307)--(4.980,8.217);
\draw[gp path] (5.231,1.184)--(5.231,1.274);
\draw[gp path] (5.231,8.307)--(5.231,8.217);
\draw[gp path] (5.455,1.184)--(5.455,1.364);
\draw[gp path] (5.455,8.307)--(5.455,8.127);
\node[gp node center] at (5.455,0.814) {100};
\draw[gp path] (6.932,1.184)--(6.932,1.364);
\draw[gp path] (6.932,8.307)--(6.932,8.127);
\node[gp node center] at (6.932,0.814) {200};
\draw[gp path] (7.795,1.184)--(7.795,1.274);
\draw[gp path] (7.795,8.307)--(7.795,8.217);
\draw[gp path] (8.408,1.184)--(8.408,1.274);
\draw[gp path] (8.408,8.307)--(8.408,8.217);
\draw[gp path] (8.883,1.184)--(8.883,1.364);
\draw[gp path] (8.883,8.307)--(8.883,8.127);
\node[gp node center] at (8.883,0.814) {500};
\draw[gp path] (9.272,1.184)--(9.272,1.274);
\draw[gp path] (9.272,8.307)--(9.272,8.217);
\draw[gp path] (9.600,1.184)--(9.600,1.274);
\draw[gp path] (9.600,8.307)--(9.600,8.217);
\draw[gp path] (9.884,1.184)--(9.884,1.274);
\draw[gp path] (9.884,8.307)--(9.884,8.217);
\draw[gp path] (10.135,1.184)--(10.135,1.274);
\draw[gp path] (10.135,8.307)--(10.135,8.217);
\draw[gp path] (10.360,1.184)--(10.360,1.364);
\draw[gp path] (10.360,8.307)--(10.360,8.127);
\node[gp node center] at (10.360,0.814) {1000};
\draw[gp path] (11.836,1.184)--(11.836,1.364);
\draw[gp path] (11.836,8.307)--(11.836,8.127);
\node[gp node center] at (11.836,0.814) {2000};
\draw[gp path] (2.027,8.307)--(2.027,1.184)--(11.836,1.184)--(11.836,8.307)--cycle;
\node[gp node center,rotate=-270] at (0.295,4.745) {Computation time [s]};
\node[gp node center] at (6.931,0.259) {Size of the matrix $A$ ($m$)};
\node[gp node right] at (10.109,2.489) {Algorithm 1};
\draw[gp path] (10.330,2.489)--(11.394,2.489);
\draw[gp path] (3.979,2.630)--(5.455,3.110)--(6.932,4.282)--(8.408,5.257)--(9.884,6.079)%
  --(11.361,6.990);
\gpsetpointsize{4.00}
\gppoint{gp mark 8}{(3.979,2.630)}
\gppoint{gp mark 8}{(5.455,3.110)}
\gppoint{gp mark 8}{(6.932,4.282)}
\gppoint{gp mark 8}{(8.408,5.257)}
\gppoint{gp mark 8}{(9.884,6.079)}
\gppoint{gp mark 8}{(11.361,6.990)}
\gppoint{gp mark 8}{(10.862,2.489)}
\node[gp node right] at (10.109,2.039) {Algorithm 2};
\draw[gp path] (10.330,2.039)--(11.394,2.039);
\draw[gp path] (3.979,3.060)--(5.455,3.586)--(6.932,4.381)--(8.408,5.494)--(9.884,6.139)%
  --(11.361,7.083);
\gppoint{gp mark 6}{(3.979,3.060)}
\gppoint{gp mark 6}{(5.455,3.586)}
\gppoint{gp mark 6}{(6.932,4.381)}
\gppoint{gp mark 6}{(8.408,5.494)}
\gppoint{gp mark 6}{(9.884,6.139)}
\gppoint{gp mark 6}{(11.361,7.083)}
\gppoint{gp mark 6}{(10.862,2.039)}
\node[gp node right] at (10.109,1.589) {Algorithm 3};
\draw[gp path] (10.330,1.589)--(11.394,1.589);
\draw[gp path] (3.979,2.318)--(5.455,2.836)--(6.932,3.759)--(8.408,4.866)--(9.884,5.497)%
  --(11.361,6.267);
\gppoint{gp mark 4}{(3.979,2.318)}
\gppoint{gp mark 4}{(5.455,2.836)}
\gppoint{gp mark 4}{(6.932,3.759)}
\gppoint{gp mark 4}{(8.408,4.866)}
\gppoint{gp mark 4}{(9.884,5.497)}
\gppoint{gp mark 4}{(11.361,6.267)}
\gppoint{gp mark 4}{(10.862,1.589)}
\draw[gp path] (2.027,8.307)--(2.027,1.184)--(11.836,1.184)--(11.836,8.307)--cycle;
\gpdefrectangularnode{gp plot 1}{\pgfpoint{2.027cm}{1.184cm}}{\pgfpoint{11.836cm}{8.307cm}}
\end{tikzpicture}
}
\caption{Computation time required for the algorithm to terminate for $\kappa(A) = 10^1$.}
\label{fig:StopTime-1e1}
\scalebox{0.72}{\begin{tikzpicture}[gnuplot]
\tikzset{every node/.append style={font={\ttfamily\fontsize{12.0pt}{14.4pt}\selectfont}}}
\gpmonochromelines
\path (0.000,0.000) rectangle (12.500,8.750);
\gpcolor{color=gp lt color border}
\gpsetlinetype{gp lt border}
\gpsetdashtype{gp dt solid}
\gpsetlinewidth{1.00}
\draw[gp path] (2.027,1.184)--(2.207,1.184);
\draw[gp path] (11.836,1.184)--(11.656,1.184);
\node[gp node right] at (1.806,1.184) {$0.001$};
\draw[gp path] (2.027,1.613)--(2.117,1.613);
\draw[gp path] (11.836,1.613)--(11.746,1.613);
\draw[gp path] (2.027,2.180)--(2.117,2.180);
\draw[gp path] (11.836,2.180)--(11.746,2.180);
\draw[gp path] (2.027,2.471)--(2.117,2.471);
\draw[gp path] (11.836,2.471)--(11.746,2.471);
\draw[gp path] (2.027,2.609)--(2.207,2.609);
\draw[gp path] (11.836,2.609)--(11.656,2.609);
\node[gp node right] at (1.806,2.609) {$0.01$};
\draw[gp path] (2.027,3.037)--(2.117,3.037);
\draw[gp path] (11.836,3.037)--(11.746,3.037);
\draw[gp path] (2.027,3.604)--(2.117,3.604);
\draw[gp path] (11.836,3.604)--(11.746,3.604);
\draw[gp path] (2.027,3.895)--(2.117,3.895);
\draw[gp path] (11.836,3.895)--(11.746,3.895);
\draw[gp path] (2.027,4.033)--(2.207,4.033);
\draw[gp path] (11.836,4.033)--(11.656,4.033);
\node[gp node right] at (1.806,4.033) {$0.1$};
\draw[gp path] (2.027,4.462)--(2.117,4.462);
\draw[gp path] (11.836,4.462)--(11.746,4.462);
\draw[gp path] (2.027,5.029)--(2.117,5.029);
\draw[gp path] (11.836,5.029)--(11.746,5.029);
\draw[gp path] (2.027,5.320)--(2.117,5.320);
\draw[gp path] (11.836,5.320)--(11.746,5.320);
\draw[gp path] (2.027,5.458)--(2.207,5.458);
\draw[gp path] (11.836,5.458)--(11.656,5.458);
\node[gp node right] at (1.806,5.458) {$1$};
\draw[gp path] (2.027,5.887)--(2.117,5.887);
\draw[gp path] (11.836,5.887)--(11.746,5.887);
\draw[gp path] (2.027,6.454)--(2.117,6.454);
\draw[gp path] (11.836,6.454)--(11.746,6.454);
\draw[gp path] (2.027,6.744)--(2.117,6.744);
\draw[gp path] (11.836,6.744)--(11.746,6.744);
\draw[gp path] (2.027,6.882)--(2.207,6.882);
\draw[gp path] (11.836,6.882)--(11.656,6.882);
\node[gp node right] at (1.806,6.882) {$10$};
\draw[gp path] (2.027,7.311)--(2.117,7.311);
\draw[gp path] (11.836,7.311)--(11.746,7.311);
\draw[gp path] (2.027,7.878)--(2.117,7.878);
\draw[gp path] (11.836,7.878)--(11.746,7.878);
\draw[gp path] (2.027,8.169)--(2.117,8.169);
\draw[gp path] (11.836,8.169)--(11.746,8.169);
\draw[gp path] (2.027,8.307)--(2.207,8.307);
\draw[gp path] (11.836,8.307)--(11.656,8.307);
\node[gp node right] at (1.806,8.307) {$100$};
\draw[gp path] (2.027,1.184)--(2.027,1.364);
\draw[gp path] (2.027,8.307)--(2.027,8.127);
\node[gp node center] at (2.027,0.814) {20};
\draw[gp path] (2.891,1.184)--(2.891,1.274);
\draw[gp path] (2.891,8.307)--(2.891,8.217);
\draw[gp path] (3.503,1.184)--(3.503,1.274);
\draw[gp path] (3.503,8.307)--(3.503,8.217);
\draw[gp path] (3.979,1.184)--(3.979,1.364);
\draw[gp path] (3.979,8.307)--(3.979,8.127);
\node[gp node center] at (3.979,0.814) {50};
\draw[gp path] (4.367,1.184)--(4.367,1.274);
\draw[gp path] (4.367,8.307)--(4.367,8.217);
\draw[gp path] (4.695,1.184)--(4.695,1.274);
\draw[gp path] (4.695,8.307)--(4.695,8.217);
\draw[gp path] (4.980,1.184)--(4.980,1.274);
\draw[gp path] (4.980,8.307)--(4.980,8.217);
\draw[gp path] (5.231,1.184)--(5.231,1.274);
\draw[gp path] (5.231,8.307)--(5.231,8.217);
\draw[gp path] (5.455,1.184)--(5.455,1.364);
\draw[gp path] (5.455,8.307)--(5.455,8.127);
\node[gp node center] at (5.455,0.814) {100};
\draw[gp path] (6.932,1.184)--(6.932,1.364);
\draw[gp path] (6.932,8.307)--(6.932,8.127);
\node[gp node center] at (6.932,0.814) {200};
\draw[gp path] (7.795,1.184)--(7.795,1.274);
\draw[gp path] (7.795,8.307)--(7.795,8.217);
\draw[gp path] (8.408,1.184)--(8.408,1.274);
\draw[gp path] (8.408,8.307)--(8.408,8.217);
\draw[gp path] (8.883,1.184)--(8.883,1.364);
\draw[gp path] (8.883,8.307)--(8.883,8.127);
\node[gp node center] at (8.883,0.814) {500};
\draw[gp path] (9.272,1.184)--(9.272,1.274);
\draw[gp path] (9.272,8.307)--(9.272,8.217);
\draw[gp path] (9.600,1.184)--(9.600,1.274);
\draw[gp path] (9.600,8.307)--(9.600,8.217);
\draw[gp path] (9.884,1.184)--(9.884,1.274);
\draw[gp path] (9.884,8.307)--(9.884,8.217);
\draw[gp path] (10.135,1.184)--(10.135,1.274);
\draw[gp path] (10.135,8.307)--(10.135,8.217);
\draw[gp path] (10.360,1.184)--(10.360,1.364);
\draw[gp path] (10.360,8.307)--(10.360,8.127);
\node[gp node center] at (10.360,0.814) {1000};
\draw[gp path] (11.836,1.184)--(11.836,1.364);
\draw[gp path] (11.836,8.307)--(11.836,8.127);
\node[gp node center] at (11.836,0.814) {2000};
\draw[gp path] (2.027,8.307)--(2.027,1.184)--(11.836,1.184)--(11.836,8.307)--cycle;
\node[gp node center,rotate=-270] at (0.295,4.745) {Computation time [s]};
\node[gp node center] at (6.931,0.259) {Size of the matrix $A$ ($m$)};
\node[gp node right] at (10.109,2.489) {Algorithm 1};
\draw[gp path] (10.330,2.489)--(11.394,2.489);
\draw[gp path] (3.979,2.996)--(5.455,3.526)--(6.932,4.544)--(8.408,5.627)--(9.884,6.302)%
  --(11.361,7.173);
\gpsetpointsize{4.00}
\gppoint{gp mark 8}{(3.979,2.996)}
\gppoint{gp mark 8}{(5.455,3.526)}
\gppoint{gp mark 8}{(6.932,4.544)}
\gppoint{gp mark 8}{(8.408,5.627)}
\gppoint{gp mark 8}{(9.884,6.302)}
\gppoint{gp mark 8}{(11.361,7.173)}
\gppoint{gp mark 8}{(10.862,2.489)}
\node[gp node right] at (10.109,2.039) {Algorithm 2};
\draw[gp path] (10.330,2.039)--(11.394,2.039);
\draw[gp path] (3.979,3.522)--(5.455,4.108)--(6.932,4.900)--(8.408,5.973)--(9.884,6.622)%
  --(11.361,7.533);
\gppoint{gp mark 6}{(3.979,3.522)}
\gppoint{gp mark 6}{(5.455,4.108)}
\gppoint{gp mark 6}{(6.932,4.900)}
\gppoint{gp mark 6}{(8.408,5.973)}
\gppoint{gp mark 6}{(9.884,6.622)}
\gppoint{gp mark 6}{(11.361,7.533)}
\gppoint{gp mark 6}{(10.862,2.039)}
\node[gp node right] at (10.109,1.589) {Algorithm 3};
\draw[gp path] (10.330,1.589)--(11.394,1.589);
\draw[gp path] (3.979,2.769)--(5.455,3.246)--(6.932,4.110)--(8.408,5.122)--(9.884,5.893)%
  --(11.361,6.689);
\gppoint{gp mark 4}{(3.979,2.769)}
\gppoint{gp mark 4}{(5.455,3.246)}
\gppoint{gp mark 4}{(6.932,4.110)}
\gppoint{gp mark 4}{(8.408,5.122)}
\gppoint{gp mark 4}{(9.884,5.893)}
\gppoint{gp mark 4}{(11.361,6.689)}
\gppoint{gp mark 4}{(10.862,1.589)}
\draw[gp path] (2.027,8.307)--(2.027,1.184)--(11.836,1.184)--(11.836,8.307)--cycle;
\gpdefrectangularnode{gp plot 1}{\pgfpoint{2.027cm}{1.184cm}}{\pgfpoint{11.836cm}{8.307cm}}
\end{tikzpicture}
}
\caption{Computation time required for the algorithm to terminate for $\kappa(A) = 10^4$.}
\label{fig:StopTime-1e4}
\scalebox{0.72}{\begin{tikzpicture}[gnuplot]
\tikzset{every node/.append style={font={\ttfamily\fontsize{12.0pt}{14.4pt}\selectfont}}}
\gpmonochromelines
\path (0.000,0.000) rectangle (12.500,8.750);
\gpcolor{color=gp lt color border}
\gpsetlinetype{gp lt border}
\gpsetdashtype{gp dt solid}
\gpsetlinewidth{1.00}
\draw[gp path] (2.027,1.184)--(2.207,1.184);
\draw[gp path] (11.836,1.184)--(11.656,1.184);
\node[gp node right] at (1.806,1.184) {$0.001$};
\draw[gp path] (2.027,1.613)--(2.117,1.613);
\draw[gp path] (11.836,1.613)--(11.746,1.613);
\draw[gp path] (2.027,2.180)--(2.117,2.180);
\draw[gp path] (11.836,2.180)--(11.746,2.180);
\draw[gp path] (2.027,2.471)--(2.117,2.471);
\draw[gp path] (11.836,2.471)--(11.746,2.471);
\draw[gp path] (2.027,2.609)--(2.207,2.609);
\draw[gp path] (11.836,2.609)--(11.656,2.609);
\node[gp node right] at (1.806,2.609) {$0.01$};
\draw[gp path] (2.027,3.037)--(2.117,3.037);
\draw[gp path] (11.836,3.037)--(11.746,3.037);
\draw[gp path] (2.027,3.604)--(2.117,3.604);
\draw[gp path] (11.836,3.604)--(11.746,3.604);
\draw[gp path] (2.027,3.895)--(2.117,3.895);
\draw[gp path] (11.836,3.895)--(11.746,3.895);
\draw[gp path] (2.027,4.033)--(2.207,4.033);
\draw[gp path] (11.836,4.033)--(11.656,4.033);
\node[gp node right] at (1.806,4.033) {$0.1$};
\draw[gp path] (2.027,4.462)--(2.117,4.462);
\draw[gp path] (11.836,4.462)--(11.746,4.462);
\draw[gp path] (2.027,5.029)--(2.117,5.029);
\draw[gp path] (11.836,5.029)--(11.746,5.029);
\draw[gp path] (2.027,5.320)--(2.117,5.320);
\draw[gp path] (11.836,5.320)--(11.746,5.320);
\draw[gp path] (2.027,5.458)--(2.207,5.458);
\draw[gp path] (11.836,5.458)--(11.656,5.458);
\node[gp node right] at (1.806,5.458) {$1$};
\draw[gp path] (2.027,5.887)--(2.117,5.887);
\draw[gp path] (11.836,5.887)--(11.746,5.887);
\draw[gp path] (2.027,6.454)--(2.117,6.454);
\draw[gp path] (11.836,6.454)--(11.746,6.454);
\draw[gp path] (2.027,6.744)--(2.117,6.744);
\draw[gp path] (11.836,6.744)--(11.746,6.744);
\draw[gp path] (2.027,6.882)--(2.207,6.882);
\draw[gp path] (11.836,6.882)--(11.656,6.882);
\node[gp node right] at (1.806,6.882) {$10$};
\draw[gp path] (2.027,7.311)--(2.117,7.311);
\draw[gp path] (11.836,7.311)--(11.746,7.311);
\draw[gp path] (2.027,7.878)--(2.117,7.878);
\draw[gp path] (11.836,7.878)--(11.746,7.878);
\draw[gp path] (2.027,8.169)--(2.117,8.169);
\draw[gp path] (11.836,8.169)--(11.746,8.169);
\draw[gp path] (2.027,8.307)--(2.207,8.307);
\draw[gp path] (11.836,8.307)--(11.656,8.307);
\node[gp node right] at (1.806,8.307) {$100$};
\draw[gp path] (2.027,1.184)--(2.027,1.364);
\draw[gp path] (2.027,8.307)--(2.027,8.127);
\node[gp node center] at (2.027,0.814) {20};
\draw[gp path] (2.891,1.184)--(2.891,1.274);
\draw[gp path] (2.891,8.307)--(2.891,8.217);
\draw[gp path] (3.503,1.184)--(3.503,1.274);
\draw[gp path] (3.503,8.307)--(3.503,8.217);
\draw[gp path] (3.979,1.184)--(3.979,1.364);
\draw[gp path] (3.979,8.307)--(3.979,8.127);
\node[gp node center] at (3.979,0.814) {50};
\draw[gp path] (4.367,1.184)--(4.367,1.274);
\draw[gp path] (4.367,8.307)--(4.367,8.217);
\draw[gp path] (4.695,1.184)--(4.695,1.274);
\draw[gp path] (4.695,8.307)--(4.695,8.217);
\draw[gp path] (4.980,1.184)--(4.980,1.274);
\draw[gp path] (4.980,8.307)--(4.980,8.217);
\draw[gp path] (5.231,1.184)--(5.231,1.274);
\draw[gp path] (5.231,8.307)--(5.231,8.217);
\draw[gp path] (5.455,1.184)--(5.455,1.364);
\draw[gp path] (5.455,8.307)--(5.455,8.127);
\node[gp node center] at (5.455,0.814) {100};
\draw[gp path] (6.932,1.184)--(6.932,1.364);
\draw[gp path] (6.932,8.307)--(6.932,8.127);
\node[gp node center] at (6.932,0.814) {200};
\draw[gp path] (7.795,1.184)--(7.795,1.274);
\draw[gp path] (7.795,8.307)--(7.795,8.217);
\draw[gp path] (8.408,1.184)--(8.408,1.274);
\draw[gp path] (8.408,8.307)--(8.408,8.217);
\draw[gp path] (8.883,1.184)--(8.883,1.364);
\draw[gp path] (8.883,8.307)--(8.883,8.127);
\node[gp node center] at (8.883,0.814) {500};
\draw[gp path] (9.272,1.184)--(9.272,1.274);
\draw[gp path] (9.272,8.307)--(9.272,8.217);
\draw[gp path] (9.600,1.184)--(9.600,1.274);
\draw[gp path] (9.600,8.307)--(9.600,8.217);
\draw[gp path] (9.884,1.184)--(9.884,1.274);
\draw[gp path] (9.884,8.307)--(9.884,8.217);
\draw[gp path] (10.135,1.184)--(10.135,1.274);
\draw[gp path] (10.135,8.307)--(10.135,8.217);
\draw[gp path] (10.360,1.184)--(10.360,1.364);
\draw[gp path] (10.360,8.307)--(10.360,8.127);
\node[gp node center] at (10.360,0.814) {1000};
\draw[gp path] (11.836,1.184)--(11.836,1.364);
\draw[gp path] (11.836,8.307)--(11.836,8.127);
\node[gp node center] at (11.836,0.814) {2000};
\draw[gp path] (2.027,8.307)--(2.027,1.184)--(11.836,1.184)--(11.836,8.307)--cycle;
\node[gp node center,rotate=-270] at (0.295,4.745) {Computation time [s]};
\node[gp node center] at (6.931,0.259) {Size of the matrix $A$ ($m$)};
\node[gp node right] at (10.109,2.489) {Algorithm 1};
\draw[gp path] (10.330,2.489)--(11.394,2.489);
\draw[gp path] (3.979,3.433)--(5.455,3.937)--(6.932,4.889)--(8.408,5.869)--(9.884,6.596)%
  --(11.361,7.973);
\gpsetpointsize{4.00}
\gppoint{gp mark 8}{(3.979,3.433)}
\gppoint{gp mark 8}{(5.455,3.937)}
\gppoint{gp mark 8}{(6.932,4.889)}
\gppoint{gp mark 8}{(8.408,5.869)}
\gppoint{gp mark 8}{(9.884,6.596)}
\gppoint{gp mark 8}{(11.361,7.973)}
\gppoint{gp mark 8}{(10.862,2.489)}
\node[gp node right] at (10.109,2.039) {Algorithm 2};
\draw[gp path] (10.330,2.039)--(11.394,2.039);
\draw[gp path] (3.979,3.557)--(5.455,4.072)--(6.932,4.915)--(8.408,6.071)--(9.884,6.624)%
  --(11.361,8.010);
\gppoint{gp mark 6}{(3.979,3.557)}
\gppoint{gp mark 6}{(5.455,4.072)}
\gppoint{gp mark 6}{(6.932,4.915)}
\gppoint{gp mark 6}{(8.408,6.071)}
\gppoint{gp mark 6}{(9.884,6.624)}
\gppoint{gp mark 6}{(11.361,8.010)}
\gppoint{gp mark 6}{(10.862,2.039)}
\node[gp node right] at (10.109,1.589) {Algorithm 3};
\draw[gp path] (10.330,1.589)--(11.394,1.589);
\draw[gp path] (3.979,3.077)--(5.455,3.593)--(6.932,4.409)--(8.408,5.509)--(9.884,6.091)%
  --(11.361,7.008);
\gppoint{gp mark 4}{(3.979,3.077)}
\gppoint{gp mark 4}{(5.455,3.593)}
\gppoint{gp mark 4}{(6.932,4.409)}
\gppoint{gp mark 4}{(8.408,5.509)}
\gppoint{gp mark 4}{(9.884,6.091)}
\gppoint{gp mark 4}{(11.361,7.008)}
\gppoint{gp mark 4}{(10.862,1.589)}
\draw[gp path] (2.027,8.307)--(2.027,1.184)--(11.836,1.184)--(11.836,8.307)--cycle;
\gpdefrectangularnode{gp plot 1}{\pgfpoint{2.027cm}{1.184cm}}{\pgfpoint{11.836cm}{8.307cm}}
\end{tikzpicture}
}
\caption{Computation time required for the algorithm to terminate for $\kappa(A) = 10^7$.}
\label{fig:StopTime-1e7}
\end{figure}
Figures~\ref{fig:StopTime-1e1}--\ref{fig:StopTime-1e7} show the time
required for each algorithm to terminate as the matrix dimension varies.
Algorithm~\ref{alg:new-adaptive-error-bound} terminates fastest, followed by Algorithm~\ref{alg:adaptive-Tatsuoka}; Algorithm~\ref{alg:new-adaptive}
takes the longest.  This is consistent with the time-error plots:
although Algorithm~\ref{alg:new-adaptive} reduces the error rapidly, its
successive-approximation stopping test may fail to recognize that the approximation
is already sufficiently accurate.  If the stopping test is not
satisfied, the algorithm doubles $n$ and continues refining.  Since
$M=2N+1$ with $N=\lceil n\arsinh(\pi n)\rceil$, this refinement
increases the number of quadrature nodes by more than a factor of two,
which can increase the time to termination.


\section{Concluding remarks}
\label{sec:conclusion}

As a method for computing $\log(A)$,
Algorithm~\ref{alg:adaptive-Tatsuoka} of Tatsuoka et al.~\cite{TatsuEn},
which is based on the DE formula, is attractive for parallel computation because
the integrand evaluations at different quadrature nodes can be
carried out independently.  When the truncation error is negligible,
its quadrature error decreases exponentially with the number $M$ of nodes.
Tatsuoka et al.~\cite{Tatsu} report that, for normalized Hermitian
positive definite matrices, the DE formula converges faster than
Gauss--Legendre quadrature when the condition number is sufficiently
large (approximately $\kappa(A)\gtrsim 2.7\times10^3$).  However, the interval
selection in their practical algorithm requires matrix norm evaluations,
which can be costly.  Moreover, because the integration interval is fixed,
the truncation error does not decrease as $M$ increases; thus, this
formulation does not theoretically guarantee convergence to $\log(A)$ as
$M\to\infty$.

To address these limitations, we proposed two algorithms based on the DE
formula.  Algorithm~\ref{alg:new-adaptive} does not require matrix norm
computations or eigenvalue information, and is therefore applicable when
such information is unavailable or expensive to obtain.  In contrast,
Algorithm~\ref{alg:new-adaptive-error-bound} uses an error bound
computable from the extreme eigenvalues of $A$ to determine the number of
quadrature nodes before evaluating the integrand, avoiding successive
refinement.
The main error bound for Hermitian positive definite matrices is given
in Theorem~\ref{thm:DE_logA_error_simple};
Theorem~\ref{thm:DE_logA_error} gives its condition-number-based form
under the normalization $\lambda_{\min}\lambda_{\max}=1$.  We further
extend the overall-error bounds to diagonalizable matrices with positive
real eigenvalues in Theorems~\ref{thm:DE_logA_error_simple_diagonalizable}
and~\ref{thm:DE_logA_error_diagonalizable}.  This extension concerns
the error analysis; the error-bound-based algorithm (Algorithm~\ref{alg:new-adaptive-error-bound})
is formulated for Hermitian positive definite matrices.  The convergence analyses
show that the proposed methods have slightly slower error decay with
respect to $M$ than Algorithm~\ref{alg:adaptive-Tatsuoka}.  Nevertheless,
in the numerical experiments, they achieve lower errors at comparable
elapsed times over the plotted ranges.

The error bound used by Algorithm~\ref{alg:new-adaptive-error-bound}
controls the quadrature error but does not account for floating-point
rounding.  In particular, for the tested matrix with
$\kappa(A)=10^7$, the observed error slightly exceeds the reported bound.
An important direction for future work is to incorporate rounding errors
and validated evaluation of the bound into a rigorous accuracy guarantee,
and to develop a numerical library that provides such certified
computations.

\appendix
\renewcommand{\thetheorem}{\Alph{section}.\arabic{theorem}}
\section{An auxiliary bound for tail trapezoidal sums}
\label{app:tail_trapezoidal_bound}

This appendix provides an auxiliary estimate used in
Section~\ref{subsec:total_error} to justify the bound on the tail
trapezoidal sums. The lemma and proof are supplied here as a supplement
to the analysis of Tatsuoka et al., rather than as a result
attributed to them.

\begin{lemma}\label{lem:tail_trapezoidal_bound}
Let $A\ne I$ be Hermitian positive definite, and fix finite $l<r$.
Define $T$ and $T_h$ by
\begin{align*}
T&=\int_{-\infty}^{l}F_{\mathrm{DE}}(x)\D x
+\int_r^{\infty}F_{\mathrm{DE}}(x)\D x,\\
T_h&=h\sum_{k=-\infty}^{-1}F_{\mathrm{DE}}(l+kh)
+h\sum_{k=1}^{\infty}F_{\mathrm{DE}}(r+kh)
+\frac h2\{F_{\mathrm{DE}}(l)+F_{\mathrm{DE}}(r)\},
\qquad h>0.
\end{align*}
Then, $\|T\|_2\ne0$, and there exists $h_0>0$, depending on $A$, $l$,
and $r$, such that, for every $0<h<h_0$,
\[
\|T_h-T\|_2\le\|T\|_2,
\qquad
\|T_h\|_2\le2\|T\|_2.
\]
\end{lemma}

\begin{proof}
Since $A\ne I$ is Hermitian positive definite, it has an eigenvalue
$\lambda\ne1$. For real $x$, the denominator in the definition of
$f_{\mathrm{DE}}(x,\lambda)$ is positive, so this function has the
constant nonzero sign of $\lambda-1$.
Thus, its two tail integrals have the same nonzero sign.
By unitary diagonalization, their sum is an eigenvalue of $T$,
which proves $\|T\|_2\ne0$.

Next, the explicit expression for $F_{\mathrm{DE}}$ shows that it and
its first derivative decay double-exponentially on the real axis.
In particular, the tail sums converge absolutely, and
the following quantity
\[
K=\frac12\left\{
\int_{-\infty}^{l}\|F_{\mathrm{DE}}'(x)\|_2\D x
+\int_r^{\infty}\|F_{\mathrm{DE}}'(x)\|_2\D x
\right\}
\]
is finite.
For any real $u$, integration by parts gives
\[
\frac h2\{F_{\mathrm{DE}}(u)+F_{\mathrm{DE}}(u+h)\}
-\int_u^{u+h}F_{\mathrm{DE}}(x)\D x
=\int_u^{u+h}\left(x-u-\frac h2\right)F_{\mathrm{DE}}'(x)\D x.
\]
Taking norms and summing over the mesh intervals in both tails yields
\[
\|T_h-T\|_2\le Kh\longrightarrow0\qquad(h\to0).
\]
By the preceding argument ($\|T\|_2\neq 0$), we have $\|T\|_2>0$. Therefore, there exists $h_0>0$ such that
$Kh\le\|T\|_2$ for $0<h<h_0$. The triangle inequality then gives
\[
\|T_h\|_2\le\|T\|_2+\|T_h-T\|_2
\le2\|T\|_2,
\]
as required.
\end{proof}

\section*{Declaration of competing interest}

The authors declare no known competing financial interests or personal
relationships that could have appeared to influence the work reported in
this paper.

\section*{Data availability}

The source code and numerical data supporting the findings of this study are
publicly available in the GitHub repository at
\url{https://github.com/okayamat/matrix-log-DE}.

\section*{Declaration of generative AI and AI-assisted technologies in the manuscript preparation process}

During the preparation of this manuscript, the authors used OpenAI Codex
to assist with language editing, consistency checks, and LaTeX
cross-reference checks. The authors reviewed
and revised the suggested changes and take full responsibility for the
final content.



\bibliographystyle{plainnat}
\bibliography{MatrixLogDE}



%
%
%
\end{document}